\batchmode
\documentclass{amsart}
\usepackage{xspace}
\usepackage{amssymb}
\usepackage{latexsym}
\usepackage{amsmath}
\usepackage{amsfonts}
\usepackage{amsthm}
\usepackage{graphicx}
\usepackage{psfrag}
\usepackage{epsfig}

\numberwithin{equation}{section}

\newtheorem{thm}{Theorem}[section]
\newtheorem{lem}[thm]{Lemma}
\newtheorem{prop}[thm]{Proposition}
\newtheorem{cor}[thm]{Corollary}

\newtheorem{mydef}[thm]{Definition}
\newcommand{\T}{\operatorname{T}}
\newcommand{\h}{\operatorname{ht}}

\newcommand{\R}{\mathbb R}
\newcommand{\SL}{\operatorname{SL}}
\newcommand{\Z}{ \mathbb{Z}}

\begin{document}
\title{Entropy and Escape of Mass for $\SL_{3}( \Z) \backslash\SL_3(\R)$}
\author{Manfred Einsiedler and Shirali Kadyrov}
\thanks{M.E. acknowledges the support of the NSF from the grant
0554373, and both authors acknowledge support by the SNF
(200021-127145).}

\begin{abstract}
We study the relation between measure theoretic entropy and escape of mass for the
case of a singular diagonal flow on the moduli space of three-dimensional unimodular lattices.
\end{abstract}
\maketitle

\section{Introduction}
Given a sequence of probability measures $\{\mu_i\}_{i=1}^\infty$ on a homogeneous space $X$, it is natural to ask what we can say about weak$^*$ limits of this sequence? Often one is interested in measures that are invariant under a transformation $T$ acting on $X$, and in this case weak$^*$ limits are clearly also invariant under $T$. If $X$ is non-compact, maybe the next question to ask is whether any weak$^*$ limit is a probability measure.
 If $T$ acts on $X=\Gamma\backslash G$ by a unipotent element where $G$ is a Lie group and $\Gamma$ is a lattice,
 then it is known that $\mu$ is either the zero measure or a probability measure
 \cite{MozSha}. This fact relies on the quantitative non-divergences estimates for
 unipotents due to works of S.~G.~Dani~\cite{Dani} (further refined by G.~A.~Margulis and D.~Kleinbock~\cite{MarKle}). On the other hand, if $T$ acts on $X=\SL_{d}( \Z) \backslash\SL_d(\R)$ by a diagonal element, then $\mu(X)$ can be any value in the interval $[0,1]$
 due to softness of Anosov-flows, see for instance \cite{Const}. However, as we will see there are constraints on $\mu(X)$ if we have additional
 information about the entropies $h_{\mu_i}(\T)$. This has been observed in \cite{TO} for the action of the geodesic flow on $\SL_2(\Z) \backslash \SL_2(\R)$, see Theorem~\ref{thm:torusorbit}.
 In this paper we will generalize this theorem to the space $\SL_{3}( \Z) \backslash \SL_{3} (\R)$ with the action of a
 particular diagonal element.

We identify $X=\SL_{d}( \Z) \backslash \SL_{d} (\R)$ with the space of unimodular lattices in $\R^{d}$, see $\S$~\ref{sec:UnModLat}.
Using this identification we can define for $d=3$ the height function $\h(x)$ of a lattice $x \in X$ as follows.
\begin{mydef} For any 3-lattice $x \in \SL_3(\Z) \backslash \SL_3(\R) $ we define the {\rm height} $\h(x)$ to be the inverse of the minimum of the length of the shortest nonzero vector in $x$ and the smallest covolume of planes w.r.t. $x$.
\end{mydef} Here, the length of a vector is given in terms of the Euclidean norm on $\R^{d}$.  Also, if $d=2$ then we consider the height $\h(x)$ to be the inverse of the length of the shortest nonzero vector in $x$. Let 
$$X_{\le M}:=\{x \in X \, | \,\h(x)\le M\} \text{  and } X_{\geq M}:=\{x \in X \, | \, \h(x)\geq M\}.$$
 By Mahler's compactness criterion (see Theorem~\ref{thm:mahler})
 $X_{\leq M}$ is compact and any compact subset of $X$ is
 contained in some $X_{\leq M}$.

In \cite{TO}, M.~E., E.~Lindenstrauss, Ph.~Michel, and A.~Venkatesh give the following theorem:
\begin{thm}
\label{thm:torusorbit}
Let $X$ be the homogeneous space $\SL(2,\Z)\backslash\SL(2,\R)$, let $T$ be the time-one-map for the geodesic flow, and $\mu$ be a $T$ invariant probability measure on $X$. Then, there exists $M_0,$ such that 
 $$
  h_\mu(T)\leq 1+\frac{\log\log M}{\log M}-\frac{\mu(X_{\geq M})}{2}
 $$
for any $M\geq M_0$.
 In particular, for a sequence of $T$-invariant probability measures $\mu_i$ with entropies $h_{\mu_i}(T)\geq c$
we have that any weak$^*$ limit $\mu$ has at least $\mu(X)\geq 2c-1$ mass left. 
\end{thm}
Here, $\mu$ is a weak$^*$ limit of the sequence $\{\mu_i\}_{i=1}^\infty$ if for some subsequence $i_k$ and for all
$f \in C_c(X)$ we have $$\lim_{k\to \infty}\int_X f d\mu_{i_k} \to \int_X f d \mu.$$

The proof of Theorem \ref{thm:torusorbit} in \cite{TO} makes use
of the geometry of the upper half plane $\mathbb H$.

From now on we let $X=\SL_{3}( \Z) \backslash
\SL_{3} (\R)$ and let
$$\alpha=\left( \begin{array}{ccc} e^{1/2} & &  \\
& e^{1/2}&\\
& & e^{-1} \end{array} \right) \in \SL_{3}(\R).$$
We define the transformation $\T:X \to X$ via $\T(x)=x\alpha$. We now state the main theorem of this paper.
\begin{thm}
\label{thm:main}
Let $X$ and $\T$ be as defined above. Then there exists a function $\varphi(M)$ (which is given explicitly), with $\varphi(M) \to_{M\to \infty} 0,$ and $M_0$ such that for any $\T$-invariant probability measure $\mu$ on $X$, and any $M>M_0$, one has
 $$h_{\mu}(\T)\leq 3-\mu(X_{\geq  M })+\varphi(M).$$
\end{thm}
In this context we note that the maximal measure theoretic entropy, the entropy of $\T$ with respect to Haar measure on $X$, is $3$. This follows e.g.\ from \cite[Prop. 9.2]{MarTom}. We will see later that $\varphi(M)=O(\frac{\log \log M}{\log M})$. 

As a consequence of Theorem~\ref{thm:main} we have:
\begin{cor}
\label{cor:lim} A sequence of $\T$-invariant probability measures
$\{\mu_i\}_{i=1}^\infty$ with entropy $h_{\mu_i}(\T) \geq c$ satisfies that any
weak$^*$ limit $\mu$ has at least $\mu(X) \geq c-2$ mass left.
\end{cor}
This result is sharp in the following sense. For any $c \in (2,3)$ one can construct a sequence of probability measures $\mu_i$ with $h_{\mu_i}(\T)\to c $ as $i\to \infty $ such that any weak$^*$ limit $\mu$ has precisely $c-2$ mass left, see \cite{Const}.

Another interesting application of our method arises when we do not assume $\T$-invariance of the measures we consider. In this case, instead of entropy consideration we assume that our measures have high dimension and study the behaviour of the measure under iterates of $T$. 

 Let us consider the following subgroups of $G$
\begin{equation}
\label{eqn:U^+}U^+=\{g\in G:\alpha^{-n}g\alpha^n\to 1 \text{ as } n\to -\infty\},
\end{equation}
\begin{equation}
\label{eqn:U^-}
U^-=\{g\in G:\alpha^{-n}g\alpha^n\to 1 \text{ as } n\to \infty\},
\end{equation}
\begin{equation}
\label{eqn:C}
C=\{g\in G : g \alpha=\alpha g\}.
\end{equation}
For any $\epsilon >0$, group $H$, and $g \in H$ we write $B_{\epsilon}^H(g)$ for the $\epsilon$-ball in $H$ around $g$, see also \S~\ref{inj}. Throughout this paper we write $A \ll B$ if there exits a constant $c>0$ such that $A \le c B$. If the constant $c$ depends on $M$, then we write $A \ll_M B.$
\begin{mydef}For a probability measure $\nu$ on $X$ we say that $\nu$ has \emph{dimension at least $d$ in the unstable direction} if for any $\delta>0$ there exists $\kappa>0$ such that for any $\epsilon\in (0,\kappa)$ and for any $\eta \in (0,\kappa)$ we have
\begin{equation}
\label{eqn:unstable}
\nu(xB_{\epsilon}^{U^+}B_\eta^{U^-C})\ll_\delta \epsilon^{d-\delta} \text{ for any } x\in X.
\end{equation}
\end{mydef}
Note that the maximum value for $d$ in the definition is 2 since $U^+$ is two dimensional. The most interesting case of this definition concerns a measure $\nu$ supported on a compact subset, say $\overline{x_0B_1^{U^+}}$, of an orbit $x_0U^+$ under the unstable subgroup. In this case, \eqref{eqn:unstable} is equivalent to $\nu(x_0 u B_\epsilon^{U^+})\ll \epsilon^{d-\delta}$ for all $u \in U^+$ (which is one of the inequalities of the notion of Ahlfors regularity of dimension $d-\delta$) and for any $\delta>0$. See \cite[Chaps. 4-6]{Mat} for more information on Ahlfors regularity. 

Let us consider the following sequence of measures $\mu_n$ defined by
$$\mu_n=\frac{1}{n}\sum_{i=0}^{n-1}\T^i_*\nu$$
where $\T^i_*\nu$ is the push-forward of $\nu$ under $\T^i$. We have
\begin{thm} 
\label{thm:dim}
For a fixed $d$, let $\nu$ be a probability measure of dimension at least $d$ in the unstable direction, and let $\mu_n$ be as above. Let $\mu$ be a weak$^*$ limit of the sequence $(\mu_n)_{n\ge 1}$. Then $\mu(X) \geq \frac{3}{2}(d-\frac{4}{3})$. In other words, at least $\frac{3}{2}(d-\frac{4}{3})$ of the mass is left.
\end{thm}
In particular, if $d=2$ then the limit $\mu$ is a probability measure. In this case with a minor additional assumption on $\nu$ one in fact obtains the equidistribution result, that is, the limit measure $\mu$ is the Haar measure \cite{Shi}.

Another application of Theorem~\ref{thm:dim} is that it gives the sharp upper bound for the Hausdorff dimension of singular pairs. The exact calculation of Hausdorff dimension of singular pairs was achieved in \cite{Che}. We say that ${\bf r} \in \R^2$ is \emph{singular} if for every $\delta>0$ there exists $N_0>0$ such that for any $N>N_0$ the inequality
$$\|q {\bf r} - {\bf p}\|<\frac{\delta}{N^{1/2}}$$
admits an integer solution for $\bf p\in \Z^2$ and for $q\in \Z$ with $0<q<N$. From our results we obtain the precise upper bound for the Hausdorff dimension of the set of singular pairs; namely this dimension is at most $\frac{4}{3}.$ This gives an independent proof for this fact which was proved in \cite{Che}.
Let $x\in\SL_3(\Z)\backslash\SL_3(\R)$. Then we say $x$ is divergent if $T^n(x)$ diverges in $\SL_3(\Z)\backslash\SL_3(\R)$.
We recall (e.g.\ from \cite{Che}) that ${\bf r}$ is singular if and only if 
$$
 x_{\bf r}=\SL_3(\Z)\begin{pmatrix}1&&\\&1&\\r_1&r_2&1\end{pmatrix}
$$ 
is divergent. An equivalent formulation\footnote{Roughly speaking the additional 6 dimensions corresponding to $U^-C,$ are not as important as the 2 directions in the unstable horospherical subgroup $U^+$. The latter is parametrized by the unipotent matrix as in the definition of $x_{\bf r}$.} of the above Hausdorff dimension result (see \cite{Che}) is that the set of divergent points in  $\SL_3(\Z)\backslash\SL_3(\R)$ has Hausdorff dimension $8-\frac23=\frac43+6$.

However, we can also strengthen this observation as follows.
 A weaker requirement on points (giving rise to a larger set) would be divergence on average, which we define as follows. A point $x$ is {\em divergent on average} (under T) if the sequence of measures
\[
 \frac1N\sum_{n=0}^{N-1}\delta_{T^n(x)}
\]
converges to zero in the weak$^*$ topology, i.e.\ if the mass of the orbit --- but not necessarily the orbit itself --- escapes to infinity. 
\begin{cor}
\label{cor:sing2}
The Hausdorff dimension of the set of points that are divergent on average is also $\frac{4}{3}+6$.
\end{cor}
We finally note that the nondivergence result \cite[Theorem 3.3]{KleLinWei} is related to Theorem~\ref{thm:dim}. In fact, \cite[Theorem 3.3]{KleLinWei} implies that $\mu$ as in Theorem~\ref{thm:dim} is a probability measure if $\nu$ has the additional regularity property; namely if $\nu$ is assumed to be friendly. However, to our knowledge these additional assumptions make it impossible to derive e.g. Corollary~\ref{cor:sing2}.

The next section below has some basic definitions and facts. In
\S~\ref{sec:marked}, we characterize what it means for a
trajectory of a lattice to be above height $M$ in some time
interval. Using this we prove Theorem~\ref{thm:main} in
\S~\ref{sec:mainprop}-\ref{sec:mainthm}. Theorem~\ref{thm:dim} and its corollary are discussed in \S~\ref{sec:locdim}.

{\bf Acknowledgements:}  We would like to thank Jim Tseng for discussions and for pointing out the reference
to \cite{Fal}. We also thank the anonymous referee for his detailed report and his suggestions.
\section{Preliminaries}
\label{sec:prelim}
\subsection{The space of unimodular lattices}
\label{sec:UnModLat}
In this section we will give a brief introduction to the space of
unimodular lattices in $\R^{3}$.
\begin{mydef}
$\Lambda \subset \R^{3}$ is a \emph{lattice}
if it is a discrete subgroup and the quotient
$\R^{3}/\Lambda$ is compact.
\end{mydef}
Note that this is equivalent to saying that
$\Lambda=\langle v_1,v_2,v_{3}\rangle_{\Z}$ where $v_1,v_2,v_{3}$ are
linearly independent vectors over $\R$.
\begin{mydef}
A lattice $\Lambda=\langle v_1,v_2,v_{3}\rangle_{\Z}$ is said to be $\emph{unimodular}$ if it has covolume equal
to $1,$ where the \emph{covolume} is the absolute value of the
determinant of the matrix with row vectors $v_1,v_2,v_{3}$.
\end{mydef}
We identify a point $\SL_{3}(\Z)g \in X$ with the unimodular lattice in
$\R^{3}$ generated by the row vectors of $g \in G$. We leave it as an exercise for the reader to convince himself that this correspondence is well defined and a bijection.

We now state Mahler's compactness criterion which motivates the definition of the height function in the introduction.
\begin{thm}[Mahler's compactness criterion]
\label{thm:mahler}
 A closed subset $K\subset X$ is compact if and
only if there exists $\delta >0$ such that no lattice in $K$ contains a nonzero vector of length less that $\delta$.
\end{thm}
For the proof the reader can refer to \cite[Corollary 10.9]{Rag}. We now deduce Corollary~\ref{cor:lim} from Theorem~\ref{thm:main}.
\begin{proof}
We need to approximate $1_{X_{\le M}}$ by functions of compact
support. So, let $f \in C_c(X)$ be such that
$$
f(x) =
\begin{cases}
1 & \text{for }x\in X_{\leq M} \\
0 & \text{for }x \in X_{\geq (M+1) }
\end{cases}
$$
and $0\leq f(x)\leq 1$ otherwise. Such $f$ exists by Urysohn's
Lemma. Hence, $$\int f \,d\mu_i \geq \int 1_{X_{\le M}}\,
d\mu_i=\mu_i(X_{\le M})\ge c-2-\varphi(M)$$
Let $\mu$ be a weak$^*$ limit, then we have
$$\lim_{i_k \to \infty}\int f \,d\mu_k =\int f\, d\mu$$
and hence we deduce that
$$ \int f \,d\mu \ge c-2-\varphi(M).$$
Now, by definition of $f$ we get $\int f \,d\mu \le \mu(X_{<(M+1)})$. Thus,
$$\mu(X_{<(M+1)})\ge c-2-\varphi(M).$$
This is true for any $M\ge M_0$, so letting $M \to \infty$ finally we have
$$\mu(X)\ge c-2$$
which completes the proof.

\end{proof}
\subsection{Riemannian metric on $ X$}
\label{inj}
Let $G=\SL_{3}(\R)$ and $\Gamma=\SL_{3}(\Z)$. We fix a left-invariant Riemannian metric $d_G$ (or simply $d$) on $G$ and for any $x_1=\Gamma g_1,x_2=\Gamma g_2 \in X$ we define
$$d_X(x_1,x_2)=\inf_{\gamma \in \Gamma} d_{G}(g_1,\gamma g_2)$$
which gives a metric $d_X$ on $X=\Gamma \backslash G.$  For more information about the Riemannian metric, we refer to \cite[Chp. 2]{RM}.

For a given subgroup $H$ of $G$ we let $B_r^H(g):=\{h \in H\,|\,d_G(h,g)<r\}$. It makes sense to abbreviate and write $B_r^H=B_r^H(1)$, where we write 1 for the identity in $G$.
\begin{mydef}
We say that $r>0$ is an injectivity radius of $x \in X$ if the map $g \mapsto xg$ from $B_r^G \to B_r^X(x)$ is an isometry.
\end{mydef}
\begin{lem}
For any $x\in X$ there exists $r>0$ which is an injectivity radius of $x$.
\end{lem}
Note that since $X_{\le M}$ is compact, we can choose $r>0$ which is an injectivity radius for every point in  $X_{\le M}$. In this case, $r$ is called \textit{an injectivity radius of $X_{\le M}$}. We refer to Proposition 9.14 in \cite{EinWar} for a proof of these claims.
\subsubsection{Operator norms}
\label{sec:metrics}
We endow $\R^3$ with the standard euclidean metric, writing $|u|$ for the norm of $u \in \R^3$. Rescaling the Riemannian metric if necessary we may assume that there exists some $\eta_0>0$ such that $|u-ug|<|u| d_G(1,g)$ for any $u \in \R^3$ and $g \in B_{\eta_0}^G$.
\subsubsection{Metric on $U^+$}
\label{sec:metric on U^+}
We may identify $U^+$ with $\R^2$ using the parametrization 
$$(t_1,t_2)\in \R^2 \to \left(\begin{array}{ccc} 1&&\\&1&\\t_1&t_2&1\end{array}\right).$$
It will be convenient to work with the maximum norm on $\R^2$. We will write $D_{\eta}^{U^+}=\{\left(\begin{array}{ccc} 1&&\\&1&\\t_1&t_2&1\end{array}\right):|t_1|,|t_2|<\eta\}$ for a ball in $U^+$ of radius $\eta$ centred at the identity. Rescaling the maximum norm on $\R^2$ if necessary we will assume that 
$$D_\epsilon^{U^+} \subset B_\epsilon^{U^+}.$$
\subsection{Entropy}
\label{sec:entropy}
Instead of giving here the formal definition of the ergodic theoretic entropy $h_\mu(\T)$ we will state only a well-known and important lemma that will enter our arguments later. We refer to \cite[\S~4]{WB} for a complete definition.

Fix $\eta>0$ small enough so that $B_{\eta}^{\SL_{3}(\R)}$ is an injective image under the exponential map of a neighborhood of $0$ in the Lie algebra.  Define a Bowen
$N$-ball to be the translate $xB_N$ for some $x \in X$ of
$$B_N=\bigcap_{n=-N}^N \alpha^{-n} B_{\eta}^{\SL_{3}(\R)}\alpha^{n}.$$
Roughly speaking the Bowen $N$-ball $x B_N$ consists of all $y $ near $x$ which have the property that the trajectories from time $-N$ to time $N$ of $x$ and $y$ are $\eta$-close to each other.

The following lemma gives an upper bound for entropy in terms of covers of Bowen balls.
\begin{lem}\label{entropy} Let $\mu$ be a $\T$-invariant probability measure on $X$. For any $N\geq
1$ and $\epsilon >0$ let $BC(N,\epsilon)$ be the minimal number of
Bowen $N$-balls needed to cover any subset of $X$ of measure
bigger that $1-\epsilon$. Then
$$h_{\mu}(\T)\leq \lim_{\epsilon \to 0}\liminf_{N \to \infty}\frac{\log
BC(N,\epsilon)}{2N}.$$
\end{lem}
We omit the proof which is very similar to \cite[Lemma 5.2]{TO} and goes back to \cite{KatBri}.
\section{Sets of labeled marked times}
\label{sec:marked}
Let $N,M>0$ be given. In this section we define for every $x \in \T^{N}(X_{\le M})$ the set of labeled marked times. Each configuration of such markings will correspond to a particular element of a partition of $X$, and we will estimate the cardinality of this partition (which is desirable due to the link of entropy and the logarithmic growth of covers as in Lemma~\ref{entropy}). This marking has the property that it will tell whether the lattice $\T^n(x)$ is above or below height $M$, without having to know $x$. However, we do not want to consider all vectors (or planes) of $x$ that become short at some point - it is likely that a partitioning of $X$ that uses all such vectors (or planes) will be too large to be of use.

Rather whenever there are two linearly independent primitive $1/M$-short vectors, our
strategy is to consider a plane in $x$ that contains both vectors.
So, for a given lattice $x$ we would like to associate a set of
labeled marked times in $[-N,N]$ which tells us when
a vector or a plane is getting resp. stops being $1/M$-short. Choosing the vectors and planes of $x$ carefully in the following construction we obtain a family $\mathcal{M}_N$ of sets of labeled marked times. This will give rise to a partition of $X$, which will be helpful in the main estimates given in \S~\ref{sec:mainprop}.
\subsection{Short lines and planes}
\label{sec:T-actiononplane}
Let $u,v \in \R^3$ be linearly independent. We recall that the covolume of the two-dimensional lattice $\Z u+\Z v$ in the plane $\R u + \R v$ equals $|u \wedge v|$. Here, $u \wedge v = (u_1,u_3,u_3)\wedge (v_1,v_2,v_3)=(u_2v_3-u_3v_2,u_3v_1-u_1v_3,u_1v_2-u_2v_1)$. Below, $u,v \in \R^3$ will always be such that $\Z u+\Z v= x \cap (\R u + \R v)$ for a lattice $x$. In this case we call $\R u + \R v$ {\it rational} w.r.t. $x$ and will call $|u \wedge v|$ {\it the covolume of the plane} $\R u + \R v$ w.r.t. $x$. We sometimes write {\it a plane $P$ in $x$} to mean the plane $P=\R u+\R v$ rational w.r.t. $x$. 

We also note that the action of $\T$ extends to $\bigwedge^2\R^2$ via
\begin{multline}
\label{eqn:exterior}
\T(u \wedge v)=(u_1e^{1/2},u_2e^{1/2},u_3e^{-1}) \wedge (v_1e^{1/2},v_2e^{1/2},v_3e^{-1})\\
=((u_2v_3-u_3v_2)e^{-1/2},(u_3v_1-u_1v_3)e^{-1/2},(u_1v_2-u_2v_1)e^1).
\end{multline}
For a plane $P=\R u+\R v$ as above, we sometimes write $\T(P)$ for $\T(u \wedge v)$. For a vector $v=(v_1,v_2,v_3) \in \R^3$ we let $\T(v):=v \alpha=(v_1e^{1/2},v_2e^{1/2},v_3 e^{-1}).$

Let $\epsilon > 0$ be given. Fix $x \in X$, a vector $v$ in $x$ is {\it $\epsilon$-short at time $n$} if $|\T^n(v)| \le \epsilon$. Similarly for plane $P \subset \R^3$ we say that it is {\it $\epsilon$-short at time $n$} (w.r.t. $x$) if $\T^n(P)$ is rational w.r.t. $\T^n(x)$ and its covolume is $\le \epsilon$.
\subsection{(Labeled) Marked Times}
\label{sec:lebMar}
For a positive number $N$ and a lattice $x \in \T^N(X_{\le M})$ we explain which times will be marked in $[-N,N]$ and how they are labeled. The following lemma which is special to $\SL_3(\Z) \backslash \SL_3(\R)$ is crucial.
\begin{lem}[Minkowski]
\label{lem:minkowski}
Let $\epsilon_1,\epsilon_2 \in(0,1)$ be given. If there are two linearly independent $\epsilon_1$-short and $\epsilon_2$-short vectors in a unimodular lattice in $x$, then there is a unique rational plane in $x$ with covolume less than $1$ which in fact is $\epsilon_1 \epsilon_2$-short.

If there are two different rational planes of covolumes $\epsilon_1$ and $\epsilon_2$ in a unimodular lattice $x$, then there is a unique primitive vector of length less than $1$ which in fact is $\epsilon_1 \epsilon_2$-short. In this case, the unique $\epsilon_1 \epsilon_2$-short vector lies in the intersection of the two short planes.
\end{lem}
The first part of the lemma follows quickly from the assumption that $x$ is unimodular. The second follows by considering the dual lattice to $x$.
We will use these facts to mark and label certain times in an efficient manner so as to keep the total number of configurations as low as possible.
\subsubsection{Some observations}
\label{sec:observations}
Let us explain how we will use Lemma~\ref{lem:minkowski}. Assume that we have the following situation:
There are two linearly independent primitive vectors $u,v$ in a unimodular lattice such that
$$|u| \le 1/M \text{ and } |\T(v)| \le 1/M.$$
Let $u=(u_1,u_2,u_3)$. It is easy to see that
$$|\T(u)|=|(e^{1/2}u_1,e^{1/2}u_2,e^{-1}u_3)| \le \frac{e^{1/2}}{ M}.$$
     Assume $M\ge e^{1/2}$. From Lemma~\ref{lem:minkowski} we have that the plane containing both $\T(u),\T(v)$ has covolume at most $\frac{e^{1/2}}{M^2}\le \frac{1}{M}$, and it is unique with this property.

The similar situation arises when we have two different planes $P,P'$ which are rational for a unimodular lattice such that
$$|P| \le 1/M \text{ and } |\T(P')| \le 1/M$$
where $| \cdot |$ means the covolume. Assume $M \ge e$. One can see that $|\T(P)| \le \frac{e}{M}.$ Thus, we conclude from Lemma~\ref{lem:minkowski} that there is a unique vector of length at most $\frac{e}{M^2}\le \frac{1}{M}$ contained in both planes $\T(P)$ and $\T(P')$.
\subsubsection{Marked times}
\label{sec:markedtimes}
Let $V_{N,x}=\{i \in [-N,N]: \T^i(x) \not \in X_{\le M}\}$. $V_{N,x}$ is a disjoint union of maximal intervals and let $V=[a,b]$ be one them. 
\begin{enumerate}
 \item[(a)] either $a=-N$ (and so $\h(T^a(x))\leq M$) or $a>-N$ and $\h(\T^{a-1}(x)) < M$, 
 \item[(b)] either $b=N$ or $\h(T^{b+1}(x))<M$, and 
 \item[(c)] $\h(\T^n(x)) \ge M$ for all $n \in V$.
 \end{enumerate}
We first show how one should inductively pick the marked times for this interval $V$:

We will successively choose vectors and planes in $x$ and mark the time instances with particular labels when these vectors and planes get $1/M$-short on $V$ and when they become big again. At time $a$ we know that there is either a unique plane or a unique vector getting $1/M$-short. Here, uniqueness of either follows from Lemma~\ref{lem:minkowski}. Moreover, we cannot have two $1/M$-short vectors ($1/M$-short planes) as otherwise there is a $1/M^2$-short plane (or vector) which contradicts the assumption that $V=[a,b]$ has $a$ as a left endpoint. If we have both a unique $1/M$-short plane and vector then we consider whichever stays $1/M$-short longer (say with preference to vectors if again this gives no decision).  Assume that we have a unique plane. The case where we start with a unique vector is similar. Mark $a$ by $p_1$ which is the time when the plane is getting $1/M$-short, and also mark by $p_1'$ the last time in $[a,b]$ when the same plane is still $1/M$-short. If $p_1'=b$ we stop marking. If not, then there is again by Lemma~\ref{lem:minkowski} a unique $1/M$-short plane or vector at $p_1'+1$. If it is a $1/M$-short plane then at time $p_1'+1$ we must have a unique $1/M$-short vector by the discussions in $\S$~\ref{sec:observations}. In either case, we have a unique $1/M$-short vector at time $p_1'+1$. Let us mark by $l_1$ the instance in $[a,p_1'+1]$ when this vector is getting $1/M$-short. Also, mark by $l_1',$ the last time in $[p_1'+1,b]$ for which this vector is still $1/M$-short. If $l_1'=b$ we stop, otherwise at time $l_1'+1$ there must be a unique $1/M$-short plane or vector. If it is a short vector then we know that there must be a unique plane of covolume at most $1/M$ by the discussions in $\S$~\ref{sec:observations}. So, in either case there is a unique $1/M$-short plane at time $l_1'+1$. So, there is an instance in $[a,l_1'+1]$ which we mark by $p_2$ when for the first time this plane is $1/M$-short. Also, mark by $p_2',$ the last instance of time in $[l_1'+1,b],$ for which the plane is $1/M$-short. If $p_2'=b$ we stop here, otherwise we repeat the arguments above and keep marking the time instances in $V$ by $l_i,l_i',p_j,p_j'$ until we hit time $b$.

Given a positive number $N$ and a lattice $x \in \T^N(X_{\le M})$ we first consider the disjoint intervals $V_i$ of maximum length with the property as $V$ above. Now start labeling some elements of the sets $V_i$ as explained earlier starting with $V_1$ and continuing with $V_2$ etc. always increasing the indices of $l_i,l_i',p_i,p_i'$. 

For any lattice $x$  as above we construct in this way a set of labeled marked times in $[-N,N]$. We denote this set by
$$\mathcal{N}(x)=\mathcal{N}_{[-N,N]}(x)=(\mathcal{L,L',P,P'}).$$
Here $\mathcal{L}=\mathcal{L}(x),\mathcal{L'}=\mathcal{L'}(x),\mathcal{P}=\mathcal{P}(x),\mathcal{P'}=\mathcal{P'}(x)$ are subsets in $[-N,N]$ that contain all the labeled marked times $l_i,l_i',p_j,p_j'$ for $x$ respectively. Finally, we let
$$\mathcal{M}_N=\{\mathcal{N}(x)\,:\, x \in \T^N(X_{\le M})\}$$
be the family of all sets of labeled marked times on the interval $[-N,N]$.
\subsubsection{The Estimates}
\begin{lem}[Noninclusion of marked intervals]
\label{lem:properties}
Let $(\mathcal{L,L',P,P'})\in \mathcal{M}_N$ be given. For any $q$ in $\mathcal{L}$ or in $\mathcal{P}$ there is no $r$ in $\mathcal{L}$ or in $\mathcal{P}$ with $q \le r \le r' \le q'$.
\end{lem}
\begin{proof}
We have four cases to consider. Let us start with the case that $r=p_i, r'=p_i'$ and $q=p_j, q'=p_j'$ (where $j>i$ as it is in our construction only possible for a later marked interval $[q, q']$ to contain an earlier one). However, by construction the plane $P_i$ that is $1/M$-short at that time we introduce the marked interval $[p_i,p_i']$ (which is either the beginning of the interval $V$ or is the time the earlier short vector stops to be short) is the unique short plane at that time. Hence, it is impossible to have the stated inclusion as the plane $P_j$ (responsible for $[p_j,p_j']$) would otherwise also be short at that time. The case of two lines is completely similar.

Consider now the case $q=p_j \in \mathcal P$ and $r=l_i \in \mathcal L$ with $p_j \le l_i \le l_i' \le p_j'$. If $l_i=a$ (and so also $l_i=p_j=a$) is the left end point of interval $V=[a,b]$ in the construction, then we would have marked either $l_i,l_i'$ or $p_j,p_j'$ but not both as we agreed to start by marking the end points of the longer interval (if there is a choice). Hence, we may assume $l_i >a$ and that times $l_i, l_i'$ have been introduced after consideration of a plane with marked times $p_k, p_k'$ satisfying $l_i \le p_k'+1 \le l_i'$, in particular $j \not = k$. We now treat two cases depending on whether $p_k \ge l_i$ or not. If $p_k \ge l_i$ then $p_j \le p_k\le p_k' \le p_j'$ which is impossible by the first case. So, assume $p_k < l_i$ then we have two different planes that are $1/M$-short at time $l_i$. This implies that the vector responsible for the interval $[l_i,l_i']$ is $1/M^2$-short by Lemma~\ref{lem:minkowski}. However, this shows that the same vector is also $1/M$-short at time $l_i-1$ for $M \ge e$, which contradicts the choice of $l_i$. The case of $q=l_i \in \mathcal L$ and $r=p_j \in \mathcal P$ is similar.
\end{proof}
We would like to know that the cardinality of $\mathcal{M}_N$ can be made small (important in Lemma~\ref{entropy}) with
$M$ large. In other words, for $M$ large we would like to say that
$\lim_{N \to \infty}\frac{\log \#\mathcal{M}_N}{2N}$
can be made close to zero. The proof is based on the geometric facts in Lemma~\ref{lem:minkowski}.

Let $\mathcal{N}=(\mathcal{L,L',P,P'}) \in \mathcal{M}_N$ and let $\mathcal{L}=\{l_1,l_2,...,l_m\}$ and $\mathcal{P}=\{p_1,p_2,...,p_n\}$ be as in the construction of marked times. It is clear from the construction that $l_i'<l_{i+1}'$ for $l_i',l_{i+1}' \in \mathcal L'$. Thus from Lemma~\ref{lem:properties} we conclude that $l_i \le l_{i+1}$. Hence we have $\mathcal{L}=\{l_1\le l_2\le ...\le l_m\}$. Similarly, we must have $\mathcal{P}=\{p_1\le p_2 \le...\le p_n\}$. In fact, we have the following.
\begin{lem}[Separation of intervals]
\label{lem:card}
For any $i=1,2,...,m-1$ and for any $j=1,2,...,n-1$ we have
$$l_{i+1}-l_i > \lfloor \log M\rfloor \text{ and } p_{j+1}-p_j > \lfloor \log M\rfloor .$$
Also,
$$l_{i+1}'-l_i' > \lfloor \log M\rfloor \text{ and } p_{j+1}'-p_j' > \lfloor \log M\rfloor .$$
\end{lem}
\begin{proof}
 For $1/M$-short vectors in $\R^{3}$, considering their forward trajectories under the action of the diagonal flow $(e^{t/2},e^{t/2},e^{-t})$, we would like to know the minimum possible amount of time needed for the vector to reach size $\ge 1$. Let $v=(v_1,v_2,v_{3})$ be a vector of size $\le 1/M$ which is of size $\ge 1$ at time $t \ge 0$. We have
$$1 \le v_1^2 e^t+v_2^2 e^t+v_3^2 e^{-2t} \le (v_1^2+v_2^2+v_3^2)e^t \le \frac{e^t}{M^2}.$$
So, we have
$$t \ge \log M^2.$$
Hence, it takes more than $2\lfloor \log M\rfloor$ steps for the vector to reach size $\ge 1$. Similarly, for a vector $v=(v_1,v_2,v_{3})$ of size $\ge 1$, we calculate a lower bound for the time $t \ge 0$ when its trajectory reaches size $\le 1/M$. We have
$$\frac{1}{M^2} \ge v_1^2 e^t+v_2^2 e^t+v_3^2 e^{-2t} \ge (v_1^2+v_2^2+v_3^2)e^{-2t} \ge e^{-2t}.$$
So, we must have $t \ge \log M$ and hence it takes at least $t=\lfloor\log M\rfloor$ steps for the vector to have size $\le 1/M$.

Now, assume that $l_{i+1}-l_i \le \lfloor \log M\rfloor.$ Let $u,v$ be the vectors in $x$ that are responsible for $l_i, l_{i+1}$ respectively. That is, $u, v$ are $1/M$-short at times $l_i,l_{i+1}$ respectively but not before. Then the above arguments imply that 
\begin{equation*}
|\T^{l_i}(v)|\le 1 \text{ and } |\T^{l_{i+1}}(u)|\le 1
\end{equation*}
 so the plane $P$ containing both $u$ and $v$ is $1/M$-short at times $l_i$ and $l_{i+1}$.
 
 The covolume of $\T^n(P)$ w.r.t. $\T^n(x)$ is $\sqrt{a_1e^{n}+a_2 e^{-n/2} }$ for some nonnegative $a_1$ and $a_2$. In particular, it is a concave function of $n$ and hence the plane $P$ is $1/M$-short in $[l_i, l_{i+1}]$ (and so $l_i, l_{i+1}$ are constructed using the same $V$). From our construction  we know that $l_i' < l_{i+1}'$. By Lemma~\ref{lem:minkowski} the same plane $P$ is $1/M^2$-short on $[l_i,l_i'] \cap [l_{i+1}, l_{i+1}']$. If this intersection is non-empty, then $P$ is also $e/M^2$-short at time $l_i'+1$. As $M \ge e$ this shows that it is the unique plane that is used to mark points, say $p_k,p_k'$, after marking $l_i,l_i'$. If on the other hand $l_i' < l_{i+1}$, then we already know that $P$ is also $1/M$-short at time $l_i'+1 \in [l_i,l_{i+1}]$ and get the same conclusion as before. Therefore, $p_k \le l_i \le l_i' \le p_k'$ which is a contradiction to Lemma~\ref{lem:properties}.
 
 The proof of the remaining three cases are very similar to the arguments above and are left to the reader.
\end{proof}
Let us consider the marked points of $\mathcal{L}$ in a subinterval of length $\lfloor \log M\rfloor$ then there could be at most 1 of them. Varying $x$ while restricting ourselves to this interval of length $\lfloor \log M\rfloor$ we see that the number of possibilities to set the marked points in this interval is no more than
$\lfloor \log M\rfloor + 1.$
For $M$ large, say $M \ge e^4$, we have
$$=\lfloor \log M\rfloor +1 \le  \lfloor \log M\rfloor^{1.25}.$$
Therefore, there are
\begin{equation*}
\le  \lfloor \log M\rfloor^{1.25(\left \lfloor \frac{2N}{\lfloor \log M\rfloor} \right \rfloor + 1)} \ll_M e^{\frac {2.5 N \log \lfloor \log M\rfloor}{\lfloor \log M\rfloor}}
\end{equation*}
possible ways of choosing labeled marked points for $\mathcal{L}$ in $[-N,N]$. The same is true for $\mathcal{L',P,P'}$. Thus we have shown the following.
\begin{lem}[Estimate of $\mathcal{M}_N$] For $M \ge e^4$ we have
\label{lem:marked}
$$\#\mathcal{M}_N \ll_M e^{\frac{10 N\log \lfloor \log M\rfloor}{\lfloor \log M\rfloor}}.$$
\end{lem}
\subsection{Configurations}
Before we end this section, we need to point out another technical detail.
For our purposes, we want to study a partition element in $X_{\le M}$ corresponding to a particular set of
labeled marked times. Since $X_{\le M}$ is compact, it is sufficient for us to study an $\eta$-neighborhood of some $x_0$
 in this partition. These are the close-by lattices which have the same set of labeled marked times. We shall see that the fact that  $N(x)=N(x_0)$, for $x$ in $x_0 B_\eta^G$, gives rise to restrictions on the position of $x$ with respect to $x_0$ (see \S
\ref{sec:restrictions}). However, just knowing that $\mathcal{N}(x_0)=\mathcal{N}(x)$ will not be sufficient for the later argument.
Hence, we need to calculate how many possible ways (in terms of vectors and planes) we can have the same labeled marked times. For this purpose, we consider the following configurations.
\subsubsection{Vectors}
Let $l$ be a marked time in the first component $\mathcal L$ of the marking $\mathcal{N}(x_0)$. Let $v_0$ be the vector in $x_0$ that is responsible for $l$ in the construction of marked times for $x_0$. Let $y=\T^{l-1}(x)$ be in $\T^{l-1}(x_0) B_\eta^{\SL_3(\R)}$ with $\mathcal{N}(x)=\mathcal{N}(x_0)$ and $v$ in $x$ that is  responsible for $l$ in the construction of marked times for $x$.  Let $v' \in x_0$ be such that $\T^{l-1}(v')g=\T^{l-1}(v) $ for some $g \in B_\eta^{\SL_3(\R)}$ with $y=\T^{l-1}(x_0) g$. We want to know how many choices for $v'$ are realized by the various choices of $x$ as above. 
\begin{lem}
\label{lem:vectorconfig}
Let $\mathcal{N}(x_0)$ be given. Also, let $l \in \mathcal L=\mathcal L(x_0)$ and $v_0 \in x_0$ be the vector which is responsible for $l$. There are two possibilities: 
\begin{enumerate}
 \item[(1)] If $l$ is the end point of a maximal interval $V$ in $V_{x_0}$, then for any $x$ with $\mathcal{N}(x)=\mathcal{N}(x_0)$ and $\T^{l-1}(x)=\T^{l-1}(x_0)g,$ with $g \in B_\eta^G$, the vector $\pm v_0\alpha^{l-1}g\alpha^{-(l-1)}$ is responsible for $l$ in $\mathcal{L}(x)$. 
 \item[(2)] If not, then there are $p,p'$ in $\mathcal{P}(x_0),\mathcal{P}'(x_0)$ respectively, with $p\le l-1\le p'$, and a set $W \subset x_0$, of size $\ll \min\{e^{p'-l},e^{(l-p)/2}\}$, such that if $x$ is a lattice such that $\mathcal{N}(x)=\mathcal{N}(x_0)$, and $\T^{l-1}(x)=\T^{l-1}(x_0)g$, with $g \in B_\eta^G$, then for some $w\in W$, $w\alpha^{l-1}g\alpha^{-(l-1)}$, is the vector responsible for $l$ in $\mathcal{L}(x)$.
 \end{enumerate}
\end{lem}
\begin{proof}
To simplify the notation below we set $w_0=\T^{l-1}(v_0) \in \T^{l-1}(x_0)$, $w=\T^{l-1}(v) \in y$, and $w'=\T^{l-1}(v')=w g \in \T^{l-1}(x_0)$.

We have
$$\frac{1}{M}\le |w| \le \frac{e}{M},$$
and so
\begin{align*}
|w'| & \le |w'-w|+|w|\\
& \le |w| d(g^{-1},1)+ |w|\\
& \le e(1+ \eta)/M.
\end{align*}
Also,
\begin{align*}
|w'| &\ge |w|-|w-w'|\\
&\ge (1-\eta)/M.
\end{align*}
Together
\begin{equation}
\label{eqn:w'}
\frac{1- \eta }{M} \le |w'| \le \frac{e(1+ \eta)}{M}.
\end{equation}
Assume first that $l=a$ is the left end point of the interval $V=[a,b]$ in the construction of marked times. In this case, $w'$ and $w_0$ lie in the same line in $\R^3$. Otherwise, if they were linearly independent then the plane containing both would be $e^2 (1+\eta)/M^2$-short by Lemma~\ref{lem:minkowski}. For $M \ge 3e^2$ this is a contradiction to the assumption that $l=a$.  Since we only consider primitive vectors we only have the choice of $w'=\pm w_0$.

Now, assume that $l$ is not the left end point of the interval $V$. Then, there is a plane $P$ in $x_0$ responsible for $p,p'$ with $p \le  l-1 \le p' $ such that 
\begin{align*}
&|\T^{p-1}(P)|\ge 1/M \text{ and }|\T^{p'+1}(P)|>1/M\\
&|\T^k(P)| \le 1/M \text{ for } k \in [p,p']. 
\end{align*}
 Let us calculate how many possibilities there are for $w'\in \T^{l-1}(x_0)$. By \eqref{eqn:w'} $w'$ is in the plane $\T^{l-1}(P)$ of covolume $<1$ w.r.t. $\T^{l-1}(x_0)$ since $\T^{l-1}(x_0)$ is unimodular. Since 
$$\frac{1}{M} <  |\T^{p'+1}(P)| \text{ and } \frac{1}{M} \le |\T^{p-1}(P)|, $$
we get
$$ \max\left \{\frac{e^{-(p'-l+2)}}{ M},\frac{e^{-(l-p)/2}}{M}\right \}\le |\T^{l-1}(P)|$$
(see \S~\ref{sec:T-actiononplane} for the action of $\T$ on planes). We note that the ball of radius $r$ contains at most $\ll \max\{\frac{r^2}{A},1\}$ primitive vectors of a lattice in $\R^2$ of covolume $A$. This follows since in the case of $r$ being smaller than the second successive minima we have at most 2 primitive vectors, and if $r$ is bigger, then area considerations give $\ll \frac{r^2}{A}$ many lattice points in the $r$-ball.

We apply this for $A=|\T^{l-1}(P)|\ge \max\ \left \{\frac{e^{-(p'-l+2)}}{ M},\frac{e^{-(l-p)/2}}{M}\right \}$ and $r=\frac{(1+\eta)e}{M}$ where 
$$\frac{r^2}{A}= \frac{(1+ \eta)^2e^2/M^2 }{\max\left \{\frac{e^{-(p'-l+2)}}{ M},\frac{e^{-(l-p)/2}}{M}\right \}}\ll \min \{e^{(p'-l)},e^{(l-p)/2} \},$$
which proves the lemma.
\end{proof}
\subsubsection{Planes}
Let $p$ be a marked time in the third component $\mathcal P$ of the marking $\mathcal{N}(x_0)$. Let $P_0$ be a plane in $T^{p-1}(x_0)$ that is responsible for $p$ in the construction of marked times for $x_0$. Let $y=\T^{p-1}(x)$ be in $\T^{p-1}(x_0) B_\eta^{\SL_3(\R)}$ with $\mathcal{N}(x)=\mathcal{N}(x_0)$ and $P$ in $x$ that is  responsible for $p$ in the construction of marked times for $x$.  Let $P'$ be a plane that is rational w.r.t. $x_0$ such that $\T^{p-1}(P')g=\T^{p-1}(P) $ for some $g \in B_\eta^{\SL_3(\R)}$ with $y=\T^{p-1}(x_0) g$. We want to know how many choices for $P'$ are realized by the various choices of $x$ as above. We have two cases.
\begin{lem}
\label{lem:planeconfig}
Let $\mathcal{N}(x_0)$ be given. Also, let $p \in \mathcal P=\mathcal P(x_0)$ and $P_0$ in $x_0$ be the plane which is responsible for $p$. There are two possibilities: 
\begin{enumerate}
 \item[(1)] If $p$ is the end point of a maximal interval $V$ in $V_{x_0}$, then for any $x$ with $\mathcal{N}(x)=\mathcal{N}(x_0)$ and $\T^{p-1}(x)=\T^{p-1}(x_0)g,$ with $g \in B_\eta^G$, the plane $P_0  \alpha^{p-1}g\alpha^{-(p-1)}$ is responsible for $p$ in $\mathcal{P}(x)$. 
 \item[(2)] If not, then there are $l,l'$ in $\mathcal{L}(x_0),\mathcal{L}'(x_0)$ respectively, with $l\le p-1\le l'$, and a set of planes $W \subset x_0$, of size $\ll \min\{e^{(l'-p)/2},e^{p-l}\}$, such that if $x$ is a lattice such that $\mathcal{N}(x)=\mathcal{N}(x_0)$, and $\T^{p-1}(x)=\T^{p-1}(x_0)g$, with $g \in B_\eta^G$, then for some $P\in W$, $P \alpha^{p-1}g\alpha^{-(p-1)}$, is the plane responsible for $p$ in $\mathcal{P}(x)$.
 \end{enumerate}
\end{lem}
We will not prove the lemma since a similar argument to that giving Lemma~\ref{lem:vectorconfig} gives this lemma.
\section{Main Proposition and Restrictions}
\label{sec:mainprop}
 Fix a height $M \geq 1$. Let $N \geq 1$ and consider $\mathcal{N}=\mathcal{N}(x_0)\in \mathcal{M}_N$. Let $V=V_{x_0} \subset [-N,N]$ be as before so that for any $n\in [-N,N]$, $n \in V_{x_0}$ if and only if there is a $1/M$-short plane or a $1/M$-short vector at time $n$. Define the set
$$Z_{\le M}(\mathcal{N}):= \{x \in \T^{N}( X_{\le M}) \,|\, \mathcal{N}(x)=\mathcal{N}\}.$$
Now, we state the main proposition.

\begin{prop}\label{prop:main}
There exists a constant $c_0>0$, independent of $M$, such that the set $Z(\mathcal{N})$ can be covered by
 $\ll_M e^{6N-|V|}c_0^{\frac{18N}{\lfloor \log M\rfloor}}$ Bowen $N$-balls.
\end{prop}
In the proof of Theorem \ref{thm:main} we will consider
$$\lim_{N \to \infty}\frac{\log \#Z(\mathcal{N})}{2N}.$$
Thus, in this limit, the term arising from $c_0^{\frac{18N}{\lfloor \log M\rfloor}}$ can be made small for $M$ large since $c_0$ does not depend on $M$. So, our main consideration is the $e^{6N-|V|}$ factor. On the other hand, it is easy to see that the set $Z(\mathcal{N})$ can be covered by $\ll e^{6N}$ many Bowen $N$-balls. But this does not give any meaningful conclusion. Therefore, $e^{-|V|}$ is the factor appearing in Proposition \ref{prop:main} that leads to the conclusion of Theorem \ref{thm:main}.

In proving Proposition \ref{prop:main}, we will make use of the lemmas below which give the restrictions needed in order to get the drop in the number of Bowen $N$-balls to cover the set $Z(\mathcal{N})$.
\subsection{Restrictions of perturbations}
\label{sec:restrictions}
\subsubsection{Perturbations of vectors}
\label{sec:Perturbations of vectors}
Let $v=(v_1,v_2,v_3)$ be a vector in $\R^{3}$.
\begin{lem}\label{lem:rest}For a vector $v$ of size $\ge 1/M$, if its trajectory under the action of $\T$ stays $1/M$-short in the time
interval $[1,S]$ then we must have $\frac{v_1^2+v_2^2}{v_3^2}< 2e^{-S}.$
\end{lem}
\begin{proof}
We will prove a slightly stronger statement. For this let $\lambda_1 \ge 1$ and $\lambda_2 \le 1$ and assume that
$$\lambda_1(v_1^2+v_2^2+v_{3}^2)\ge \frac{1}{M^2}\geq \lambda_2(v_1^2e^{S}+v_2^2e^{S}+v_{3}^2e^{-2S}).$$
This simplifies to
$$\lambda_1 v_{3}^2>(v_1^2+v_2^2)(\lambda_2 e^{S}-\lambda_1).$$
Assuming $\lambda_1,\lambda_2$ are close to 1, we must have $v_{3}\neq 0$ and
$$\frac{v_1^2+v_2^2}{v_3^2}\le \frac{\lambda_1}{\lambda_2e^S-\lambda_1}.$$
Assuming again that $\lambda_1,\lambda_2$ are close to 1 the last expression is bounded by $2e^{-S}$.
\end{proof}
We would like to get restrictions for the vectors which are close to the vector $v$ and whose trajectories behave as $v$ on the time interval $[0,S]$ . So, let $u=(u_1,u_2,u_3)$ be a vector in $ \R^{3}$ with $u=v g$ for some $g \in B_\eta^{\SL_3(\R)}$ such that $|u|\ge 1/M$ and that its forward trajectory stays $1/M$-short in the time interval $[1,S]$.

Let us first assume $g=\left( \begin{array}{ccc} 1& & \\
 & 1&\\
-t_1 &-t_2 &1  \end{array} \right) \in
B_{\eta}^{U^+}$ so that
\begin{equation*}
 \left( \begin{array}{ccc}
 u_1&u_2 &u_{3}  \end{array} \right)=
  \left( \begin{array}{ccc}
 v_1&v_2 &v_3  \end{array}\right)\left( \begin{array}{ccc}  1& & \\
 &1&\\
-t_1 &-t_2 &1  \end{array} \right) .
\end{equation*}
From Lemma~\ref{lem:rest} we know that $\frac{u_1^2+u_2^2}{u_3^2}< 2e^{-S}.$ So,
$$\frac{(v_1-v_3t_1)^2+(v_2 -v_3t_2)^2}{v_3^2} < 2e^{-S}.$$
We are interested in possible restrictions on $t_j$'s since they belong to the unstable horospherical subgroup of $\SL_{3}(\R)$ under conjugation by $\alpha=\text{diag}(e^{1/2},e^{1/2},e^{-1})$. Simplifying the left hand side, we obtain
$$(\frac{v_1}{v_3}-t_1)^2+(\frac{v_2}{v_3}-t_2)^2 < 2e^{-S}.$$
We also know $\frac{v_1^2}{v_3^2}+\frac{v_2^2}{v_3^2}<2e^{-S}.$ Together with the triangular inequality, we get
$$t_1^2+t_2^2<(\sqrt{2 e^{-S}}+\sqrt{2 e^{-S}})^2=8 e^{-S}.$$
In general, we have 
$$g=\left( \begin{array}{ccc} 1& & \\
 & 1&\\
-t_1 &-t_2 &1  \end{array} \right) \left( \begin{array}{ccc} a_{11}&a_{12} &a_{13} \\
 a_{21}&a_{22}&a_{23}\\
0 & 0&a_{33}  \end{array} \right) \in B_{\eta}^{\SL_{3}(\R)}.$$
In this case, we still claim that $$t_1^2+t_2^2< 8 e^{-S}.$$
Let 
$$w=\left( \begin{array}{ccc}
 w_1&w_2 &w_{3}  \end{array} \right)=
  \left( \begin{array}{ccc}
 v_1&v_2 &v_3  \end{array}\right)\left( \begin{array}{ccc}  1& & \\
 &1&\\
-t_1 &-t_2 &1  \end{array} \right) $$
so that 
\begin{equation}
\label{eqn:u=vg}
u=v g=w \left( \begin{array}{ccc} a_{11}&a_{12} &a_{13} \\
 a_{21}&a_{22}&a_{23}\\
0 & 0&a_{33}  \end{array} \right).
\end{equation}
We observe
$$
\T^S(u)=\T^S(w) \left( \begin{array}{ccc} a_{11}&a_{12} &a_{13} e^{-3S/2}\\
 a_{21}&a_{22}&a_{23}e^{-3S/2}\\
 0&0&a_{33}\end{array}\right) ,
$$
so that $\T^S(u) \in \T^S(w) B_{\eta}^{\SL_3(\R)}$ and $|\T^S(u)-\T^S(w)|<\eta |\T^S(u)|$ by the discussion in \S~\ref{sec:metrics}.
Hence, $|\T^S(u)|<1/M$ implies
\begin{equation}
\label{eqn:|w alpha^S|}
 |\T^S(w)| \le |\T^S(u)|+|\T^S(u)-\T^S(w)|<\frac{1+\eta}{M}.
 \end{equation}
On the other hand, since $g\in B_{\eta}^{\SL_3(\R)}$ we have
\begin{equation}
\label{eqn:|w|}
|w| \ge |u|-|u-w| > \frac{1-\eta}{M}
\end{equation}
Combining \eqref{eqn:|w alpha^S|} and \eqref{eqn:|w|} we get
$$\frac{|w|}{1-\eta}>\frac{1}{M}>\frac{|\T^S(w)|}{1+\eta}.$$
Now, the proof of Lemma~\ref{lem:rest}, for sufficiently small $\eta>0$ implies
$$\frac{w_1^2+w_2^2}{w_3^2}<2 e^{-S}.$$
Hence, we are in the previous case with $u$ replaced by $w$. So, we have $t_1^2+t_2^2<8 e^{-S}$ which proves the claim. We have shown the following.
\begin{lem}
\label{lem:uv}
There exists a sufficiently small $\eta>0$, such that for any $M,S >0$ the following holds. Let $v,u$ be vectors in $\R^{3}$ with sizes $\ge 1/M$ whose trajectories in $[1,S]$ stay $1/M$-short. Assume that $u=vg$ with $g \in  B_\eta^{\SL_3(\R)}$ and that the notation is as in \eqref{eqn:u=vg}. Then
$$t_1^2+t_2^2\leq 8 e^{-S}.$$
\end{lem}
\begin{lem}
\label{lem:volv}
Let $\eta>0$ be given. For any $S,S' > 0$, let us divide $[-2\eta,2\eta]^2$ into
small squares of side length $\frac{1}{2}\eta e^{-3S'/2}$. Then there exists a constant $c>0$ such that there are  $\ll \max \{1,e^{3S'-S}\} $ small squares that intersect with the ball $t_1^2+t_2^2 \leq 8 e^{-S}$ on $[-2\eta,2\eta]^2$.
\end{lem}
\begin{proof}
Note that $t_1^2+t_2^2 \leq 8 e^{-S}$ defines a ball with diameter $2\sqrt{8}e^{-S/2}$. If $\frac{1}{2}\eta e^{-3S'/2} \ge 2 \sqrt{8} e^{-S/2}$ then there are 4 squares that intersects the ball. Otherwise (which makes $3S'-S$ bounded below), there can be at most $ \ll \frac{(e^{-S/2})^2}{(e^{-3S'/2})^2}=e^{3S'-S}$
small squares that intersect with the given ball. 
\end{proof}
What Lemma \ref{lem:uv} and Lemma \ref{lem:volv} say is the following:

Consider a neighborhood ${\it O}=x_0 B_{\eta/2}^{U^+}B_{\eta/2}^{U^-C}$ of $x_0$ in $X$ where as before  $U^+,U^-,$ and $C$ are the unstable, stable, and centralizer subgroups of $\SL_3(\R)$ with respect to $\alpha$, respectively. If we partition the square with side length $2\eta$ in $B_{\eta/2}^{U^+}$ into small squares with side lengths $\eta e^{-3S'/2}$, then we have $\ll \lceil \frac{2 \eta}{\eta e^{-3S'/2}} \rceil^2 \ll \lceil e^{3S'/2} \rceil^2$ many elements in this partition. Now, assume that there is a vector $v \in x_0$ with $|v|\ge 1/M$ that stays $1/M$-short in $[1, S]$ and consider the set of lattices $x=x_0 g$  in ${\it O}$ with the property that the vector $w=vg$ in $x$ behaves as $v$ in $[0,S]$. Then the above two lemmas say that this set is contained in $\leq c_0e^{3S'-S}$ many partition elements (small squares). Hence, in the proof of Proposition \ref{prop:main}, instead of $\leq c_0 \lceil e^{3S'/2} \rceil^2$ many Bowen balls we will only consider $\leq c_0 e^{3S'-S}$ many of them and this (together with the case below) will give us the drop in the exponent as appeared in Proposition~\ref{prop:main}.
\subsubsection{Perturbations of planes}
Assume that for a lattice $x \in X$ there is a rational plane $P$ w.r.t. $x$ with
$$|P|\ge 1/M \text{ and } |\T^k(P)| \le 1/M \text{ for } k \in [1,S].$$
Let $u,v$ be generators of $P$ with $|P|=|u \wedge v|$. So we have
$$|u\wedge v|\ge 1/M \ge |\T^{S}(u\wedge v)|.$$
Thus, substituting $a=u_2v_3-u_3v_2,b=u_3v_1-u_1v_3,c=u_1v_2-u_2v_1$ (cf.~\ref{eqn:exterior}) we obtain
$$a^2+b^2+c^2 \ge  a^2e^{-S}+b^2e^{-S}+c^2e^{2S},$$
which gives
\begin{equation*}
\frac{c^2}{a^2+b^2}\le \frac{1-e^{-S}}{e^{2S}-1}=e^{-2S}\frac{1-e^{-S}}{1-e^{-2S}}=e^{-2S}\frac{1}{1+e^{-S}}<e^{-2S}.
\end{equation*}
Assume $x'=xg$ for some $g\in B_\eta^{\SL_3(\R)}$. For now, let us assume that 
$$g=\left( \begin{array}{ccc}  1&& \\
&1&\\
t_1&t_2&1 \end{array} \right).$$
Let $u',v' \in x'$ be such that
\begin{align*}
\left( \begin{array}{c}  u' \\
v' \end{array} \right)=\left( \begin{array}{ccc}  u_1'&u_2'&u_3' \\
v_1'&v_2'&v_3' \end{array} \right)&=\left( \begin{array}{ccc}  u_1&u_2&u_3 \\
v_1&v_2&v_3 \end{array} \right)\left( \begin{array}{ccc}  1&& \\
&1&\\
t_1&t_2&1 \end{array} \right)\\
&=\left( \begin{array}{ccc}  u_1+t_1u_3&u_2+t_2u_3&u_3 \\
v_1+t_1v_3&v_2+t_2v_3&v_3 \end{array} \right).
\end{align*}
We let $a'=u_2' v_3' -u_3' v_2'=(u_2+t_2 u_3)v_3-u_3 (v_2+t_2 v_3)$ and hence $a'=a$. Similarly, $b'=u_3' v_1'-u_1' v_3'=b$ and let
\begin{multline*}
c'=u_1' v_2' - u_2'v_1'=\\
(u_1+t_1 u_3)(v_2+t_2 v_3)-(u_2+t_2 u_3)(v_1+t_1 v_3)=c-a t_1 -b t_2.
\end{multline*}
Now, assume that
\begin{align*}
|u'\wedge v'|\ge 1/M \text{ and } |\T^k(u'\wedge v')| \le 1/M \text{ for } k \in [1,S]
\end{align*}
which by the above implies
$$\frac{c'^2}{a'^2+b'^2}=\frac{(c-at_1-bt_2)^2}{a^2+b^2}<e^{-2S}.$$
For a general $g \in B_\eta^{\SL_3(\R)}$ we would like to obtain a similar equation. Let us write $g$ as
\begin{equation}
\label{eqn:g}
g=\left( \begin{array}{ccc}  1&& \\
&1&\\
t_1&t_2&1 \end{array} \right)\left( \begin{array}{ccc}  g_{11}&g_{12}&g_{13} \\
g_{21}&g_{22}&g_{23}\\
0&0&g_{33} \end{array} \right).
\end{equation}
Then we have
$$\T^l(x')=\T^l(xg)=\T^l\left(x\left( \begin{array}{ccc}  1&& \\
&1&\\
t_1&t_2&1 \end{array} \right)\right)\left( \begin{array}{ccc}  g_{11}&g_{12}&g_{13}e^{-\frac{3}{2}l} \\
g_{21}&g_{22}&g_{23}e^{-\frac{3}{2}l}\\
0&0&g_{33} \end{array} \right).$$
Hence the forward trajectories of $x'$ and $x\left( \begin{array}{ccc}  1&& \\
&1&\\
t_1&t_2&1 \end{array} \right)$ stay $\ll \eta$ close.
Thus, we have $$\frac{(c-at_1-bt_2)^2}{a^2+b^2}\ll e^{-2S}.$$
From the triangular inequality we obtain
$$\frac{(at_1+bt_2)^2}{a^2+b^2}\ll e^{-2S}.$$
Let $C>0$ be the constant that appeared in the last inequality.
\begin{lem}
\label{lem:LL'}
Let $P,P'$ be two dimensional lattices in $\R^{3}$ of covolume $\ge 1/M$ whose trajectories in $[1,S]$ stay $1/M$-short and assume that $P'=Pg$ for some $g\in B_{\eta}^{\SL_3(\R)}$, then for some $a,b$ (dependent on $P$) we must have in the notation of \eqref{eqn:g} that
$$\frac{(at_1+bt_2)^2}{a^2+b^2}\leq C e^{-2S}.$$
\end{lem}
We note that the inequality above describes a neighborhood of the line in $\R^2$ defined by the normal vector $(a,b)$ of width $2\sqrt{C}e^{-s}.$
\begin{lem}
\label{lem:volp}
Consider the set defined by $\frac{(at_1+bt_2)^2}{a^2+b^2}\leq C e^{-2S}$ on
$[-2\eta,2\eta]^2$ and let us divide $[-2\eta,2\eta]^2$ into
small squares of side length $\frac{1}{2}\eta e^{-3S'/2}$. Then there are  $\ll \max \{e^{3S'/2},e^{3S'-S}\} $ small squares that intersect with the region $\frac{(at_1+bt_2)^2}{a^2+b^2}\le C e^{-2S}.$
\end{lem}
\begin{proof}
The type of estimate depends on whether the side length $\frac{1}{2}\eta e^{-3S'/2}$ of the squares is smaller or bigger than the width $2\sqrt{C} e^{-S}$ of the neighborhood.
We need to calculate the length and the area of the region $R$ given by $$|at_1+bt_2|\leq \sqrt{C(a^2+b^2)} e^{-S}$$ restricted to $[-2\eta,2\eta]^2.$ As mentioned earlier, the inequality above describes a $\sqrt{C} e^{-S}$-neighborhood of the line $at_1+bt_2=0.$ The length of the segment of this line in $[-2\eta,2\eta]^2$ is at most $4\sqrt{2}\eta$, so that the area of $R$ is $\le 4\sqrt{2C} \eta e^{-S}$. 

If $\sqrt{C} e^{-S} \le \frac{1}{2}\eta e^{-3S'/2}$ then there are $\ll \frac{\eta}{\eta e^{-3S'/2}}= e^{3S'/2}$ many intersections. Otherwise, there are at most 
$$\ll \frac{\sqrt{C} \eta e^{-S}}{\eta^2 e^{-3S'}} \ll e^{3S'-S}$$
small squares that intersect the region $R$.
\end{proof}
\subsection{Proof of Main Proposition}
\begin{proof}[Proof of Proposition~\ref{prop:main}] 
By taking the images under a positive power of $\T$ it suffices to consider forward trajectories and the following reformulated problem:

Let $V \subset [0,N-1]$ and $x_0 \in X_{\le M}$ be such that
$$n \in V \text{ if and only if } \T^n(x_0) \in X_{\ge M}.$$
Also let $\mathcal{N}=\mathcal{N}_{[0,N-1]}(x_0)$ be the marked times for $x_0$ (defined similarly to $\mathcal N_{[-N,N]}$ as in \S~\ref{sec:markedtimes}).

We claim that 
$$Z_{\le M}^+=\{x \in X_{\le M} : \mathcal{N}_{[0,N-1]}(x)=\mathcal{N}\}$$
can be covered by $\ll_M e^{3N-|V|} c_0^{\frac{9N}{\lfloor \log M \rfloor}}$ forward Bowen $N$-balls $x B_N^+$ defined by
$$B_N^+=\bigcap_{n=0}^{N-1} \alpha^{n} B_{\eta}^{\SL_{3}(\R)}\alpha^{-n}.$$
Since $X_{\le M}$ is compact and since we allow the implicit constant above to depend on $M$ it suffices to prove the following:

Let $U^+,U^-$, and $C$ be the subgroups of $G$ introduced in \eqref{eqn:U^+}, \eqref{eqn:U^-}, and \eqref{eqn:C} respectively. Given $x_0 \in X_{\le M}$ and a neighborhood
$${\it O}=x_0 D_{\eta/2}^{U^+}B_{\eta/2}^{U^-C}$$
 of $x_0$ where as before $D_{\eta/2}^{U^+}$ is the $\eta/2$-neighborhood of $1$ in $U^+$ (identified with $\R^2$) w.r.t. maximum norm. Then we claim that the set
 $$Z_{\emph O}^+=\{x \in {\it O} : \mathcal{N}_{[0,N-1]}(x)=\mathcal{N} \}$$
can be covered by $\ll e^{3N-|V|} c_0^{\frac{9N}{\lfloor \log M \rfloor}}$ forward Bowen $N$-balls.

  If we apply $\T^n$ to ${\it O}$ we get a neighborhood of
  $\T^n(x_0)$ for which the $U^+$-part is stretched by the factor $e^{3n/2}$, while the
  second part is still in $B_{\eta/2}^{U^-C}$. By breaking the
  $U^+$-part into $\lceil e^{3n/2} \rceil^2$ sets of the form $u_i^+
  D_{\eta/2}^{U^+}$ for various $u_i^+ \in U^+$ we can write $\T^n(\emph
  O)$ as a union of $\lceil e^{3n/2} \rceil^2$ sets of the form 
  $$\T^n(x_0)u_i^+D_{\eta/2}^{U^+}\alpha^{-n} B_{\eta/2}^{U^-C}\alpha^{n}.$$
 Hence we got similar neighborhoods as before. If we take the
 pre-image under $\T^n$ of this set, we obtain the set
 $$\T^{-n}(\T^n(x_0)u_i^+)\alpha^nD_{\eta/2}^{U^+}\alpha^{-n} B_{\eta/2}^{U^-C}.$$
 Notice that $\T^{-n}(\T^n(x_0)u_i^+)\alpha^nD_{\eta/2}^{U^+}\alpha^{-n} B_{\eta/2}^{U^-C}$ is contained in the forward Bowen $n$-ball $\T^{-n}(\T^n(x_0)u_i^+)B_n^+.$ Indeed by assumption on the metrics (see \S~\ref{sec:metric on U^+}) we have $D_\epsilon \subset B_\epsilon$  and so for $0\le k < n$ we have 
$$
\alpha^{-k}(\alpha^n D_{\eta/2}^{U^+} \alpha^{-n})\alpha^k 
\subset \alpha^{n-k} B_{\eta/2}^{U^+} \alpha^{-(n-k)} \alpha^{-k}B_{\eta/2}^{U^- C} \alpha^k \subset B_{\eta/2}^{U^+}B_{\eta/2}^{U^-C}\subset B_{\eta}^{\SL_3(\R)}.
$$
 We would like to reduce the number of $u_i^+$'s, so that we do not have to
 use all $\lceil e^{3n/2} \rceil^2$ forward Bowen $n$-balls to cover the set $Z_{\it O}^{+}$.

We can decompose $V$ into maximal intervals $V_1,V_2,\dots,V_m$ for some $m$. We note here that $m \le |\mathcal{L}|+|\mathcal{P}|$ so that from Lemma~\ref{lem:card} we obtain
\begin{equation}
\label{eqn:boundform}
m \le \frac{2N}{\lfloor \log M \rfloor}+2
\end{equation}
Now, write $[0, N-1]\setminus V = W_1 \cup W_2 \cup ... \cup W_l$ where $W_i$'s
 are maximal intervals. A bound similar to \eqref{eqn:boundform} also holds for $l$.
 
 We will consider intervals $V_j$ and $W_i$ in their respective
 order in $[0, N-1]$. At each stage we will divide any of the sets
 obtained earlier into $\lceil e^{3|V_j|/2}\rceil^2$- or $\lceil e^{3|W_i|/2}\rceil^2$- many
 sets, and in the case of $V_j$ show that we do not have to keep
 all of them. We inductively prove the following:

  For $K \leq N$ such that $[0,K]=V_1 \cup V_2 \cup ... \cup V_n \cup W_1 \cup W_2 \cup ... \cup W_{n'} $ the set $Z_{{\it O}}^+$ can be covered by
 $\ll e^{3K}e^{-(|V_1|+...+|V_{n}|)} c_0^{4\frac{|V_1|+...+|V_n|}{\lfloor \log M\rfloor}+4n+n'}$
 many pre-images under $\T^K$ of sets of the form
$$\T^K(x_0)u^+D_{\eta/2}^{U^+}\alpha^{-K} B_{\eta/2}^{U^-C}\alpha^K$$
 and hence can be covered by $\ll e^{3K}e^{-(|V_1|+...+|V_{n}|)}c_0^{4\frac{|V_1|+...+|V_n|}{\lfloor \log M\rfloor}+4n+n'}$ many forward Bowen K-balls. When $K=N$ we obtain the proposition.

 For the inductive step, if the next interval is $W_{n'+1}$ then after dividing the set $\T^K(x_0)u^+B_{\eta/2}^{U^+}\alpha^{-K} B_{\eta/2}^{U^-C}\alpha^K$ into $\lceil e^{3|W_{n'+1}|/2}\rceil^2\le 4e^{3|W_{n'+1}|} $ many sets of the form 
 $$\T^{K+|W_{n'+1}|}(x_0)u^+B_{\eta/2}^{U^+}\alpha^{-K-|W_{n'+1}|} B_{\eta/2}^{U^-C}(1)\alpha^{K+|W_{n'+1}|}$$
 we just consider all of them, and hence have that $Z_{\emph O}^+$ can be covered by 
 $$\ll e^{3(K+|W_{n'+1}|)}e^{-(|V_1|+...+|V_{n}|)} c_0^{4\frac{|V_1|+...+|V_n|}{\lfloor \log M\rfloor}+4n+n'+1}$$
 many forward Bowen $K+|W_{n'+1}|$-balls (assuming $c_0 \ge 4$). 
 
  So, assume that the next time interval is
 $V_{n+1}=[K+1,K+R]$. Pick one of the sets obtained in an earlier step
 and denote it by 
 $$Y=\T^K(x_0)u^+B_{\eta/2}^{U^+}\alpha^{-K} B_{\eta/2}^{U^-C}\alpha^K.$$

We are interested in lattices $x \text{ in } Y\cap X_{\le M}$ such that 
$$\mathcal{N}_{[0,R]}(x)=\mathcal{N}_{[0,R]}(\T^K(x_0))=\{\mathcal{L,L',P,P'}\}.$$
We have
$$\mathcal{L}=\{l_1<l_2<...<l_k\},\,\, \mathcal{L'}=\{l_1'<l_2'<...<l_k'\}$$
and
$$\mathcal{P}=\{p_1<p_2<...<p_{k'}\}, \,\,\mathcal{P'}=\{p_1'<p_2'<...<p_{k'}'\}$$
for some $k,k' \ge 0$. For simplicity of notation assume that $K+1=l_1.$ We note that 
$$K+1=l_1< p_1<l_2<p_2<...<\min\{l_k,p_{k'}\}<\max\{l_k,p_{k'}\}.$$
This easily follows from the construction of labeled marked times together with Lemma~\ref{lem:properties}. So, we can divide the interval $V_{n+1}$ into subintervals 
$$[l_1,p_1],[p_1,l_2],...,[\min\{l_k,p_{k'}\},\max\{l_k,p_{k'}\}],[\max\{l_k,p_{k'}\},K+R].$$
We consider each of the (overlapping) intervals in their respective order. 

Let us define $c_0$ to be the maximum of the implicit constants that appeared in the conclusions of Lemma~\ref{lem:vectorconfig}, Lemma~\ref{lem:planeconfig}, Lemma~\ref{lem:volv}, and Lemma~\ref{lem:volp}. 

We would like to apply Lemma~\ref{lem:volv} and Lemma~\ref{lem:volp} to obtain a smaller number of forward Bowen $K+|V_{n+1}|$-balls to cover the set $\T^{-K}(Y)$. Assume for example that there is a vector $v$ in a lattice $x$ that is getting $1/M$-short and staying short in some time interval, also assume that there is a vector $u$ in a lattice $xg$ for some $g\in B_{\eta}^{\SL_3(\R)}$ which behaves the same as $v$. However, we can apply Lemma~\ref{lem:volv} only if we know that $u=vg.$ Thus, it is necessary to know how many vectors $w'$ there are in $x$ for which $u=w'g$ for some $g$. This is handled by Lemma~\ref{lem:vectorconfig}. Similar situation arises when we want to apply Lemma~\ref{lem:volp}, and this case we first need to use Lemma~\ref{lem:planeconfig}.

Let us start with the interval $[l_1,p_1]$. Let us divide the set $Y \cap X_{\le M}$ into $\lceil e^{3(p_1-l_1)/2}\rceil^2$ small sets by partitioning the set $D_{\eta/2}^{U^+}$ in the definition of $Y$ as we did before. Since $l_1$ is the left end point of $V_{n+1}$ we see that the assumptions of Lemma~\ref{lem:uv} are satisfied in the sense that if there is a lattice $\T^{l_1-1}(x_0)g$ which has the same set of marked times as $\T^{l_1-1}(x_0)$ for some $g\in B_{\eta}^{\SL_3(\R)},$ then there are unique vectors $v\in \T^{l_1-1}(x_0)$ and $u=v g \in \T^{l_1-1}(x_0)g$ which are of size $\ge 1/M$ and stay $1/M$-short in $[l_1,l_1']$. (cf. Lemma~\ref{lem:vectorconfig}). Now, from Lemma~\ref{lem:uv} and Lemma~\ref{lem:volv} with $S'=p_1-l_1$ and $S=l_1'-l_1$ we see that we only need to consider 
\begin{equation}
\label{eqn:N_1}
\le c_0\max\{1,e^{3(p_1-l_1)-(l_1'-l_1)}\}=:N_1
\end{equation}
 of these $\lceil e^{3(p_1-l_1)/2}\rceil^2$  sets (see the discussion at the end of \S~\ref{sec:Perturbations of vectors}). Thus, we obtain sets of the form 
 $$\T^{p_1}(x_0)u^+D_{\eta/2}^{U^+}\alpha^{-p_1} B_{\eta/2}^{U^-C}\alpha^{p_1}.$$
Now, let us consider the next interval $[p_1, l_2]$. Divide the sets obtained earlier into $\lceil e^{3(l_2-p_1)/2}\rceil^2$ subsets for which the $U^+$-component is of the from $u^+D_{e^{-3(l_2-p_1)/2}\eta/2}^{U^+}$. We would like to apply Lemma~\ref{lem:volp}. However, Lemma~\ref{lem:volp} concerns itself with the restrictions on $g$ arising from common behaviors of two planes $P,P'=Pg$ and we only know the common behavior of the lattices. Moreover, if $P_0$ (resp. $P$) is the plane that is rational w.r.t. $\T^{p_1}(x_0)$ (resp. $\T^{p_1}(x_0)g$) which is responsible for the marking of $[p_1,p_1']$ then we do not necessarily know that $P=P_0 g$. On the other hand, we see from Lemma~\ref{lem:planeconfig} that there are $\le c_0  \min\{e^{(l_1'-p_1)/2},e^{p_1-l_1}\}$ choices of planes $P'$ that are rational w.r.t. $\T^{p_1}(x_0)$ for which we could possibly have $P=P'g.$ For each choice we can apply Lemma~\ref{lem:volp} with $S'=l_2-p_1$ and $S=p_1'-p_1$. Thus, for each choice we need to consider only $\le c_0 \max\{e^{3(l_2-p_1)/2},e^{3(l_2-p_1)-(p_1'-p_1)}\}$ of the $\lceil e^{3(l_2-p_1)/2}\rceil^2$ subsets. Thus, in total, we need to consider only
\begin{equation}
\label{eqn:N_2}
\le c_0^2  \min\{e^{(l_1'-p_1)/2},e^{p_1-l_1}\} \max\{e^{3(l_2-p_1)/2},e^{3(l_2-p_1)-(p_1'-p_1)}\}=:N_2
\end{equation}
of these subsets. 

Taking the images of these sets under $\T^{l_2-p_1}$ we obtain sets of the form
$$\T^{l_2}(x_0)u^+D_{\eta/2}^{U^+}\alpha^{-l_2} B_{\eta/2}^{U^-C}\alpha^{l_2}.$$
Now, let us consider the interval $[l_2,p_2]$ and let us divide the sets obtained earlier into $\lceil e^{3(p_2-l_2)/2}\rceil^2$ subsets of the form
$$\T^{p_2}(x_0)u^+D_{\eta/2}^{U^+}\alpha^{-p_2} B_{\eta/2}^{U^-C}\alpha^{p_2}.$$
From Lemma~\ref{lem:vectorconfig} we know that there are $\le c_0 \min \{e^{p_1'-l_2},e^{(l_2-p_1)/2}\}$ many configurations and for each of them we can apply Lemma~\ref{lem:volv} with $S'=p_2-l_2$ and $S=l_2'-l_2$. So, for each configuration we need only $\le c_0\max\{1,e^{3(p_2-l_2)-(l_2'-l_2)}\}$ many of the subsets. Thus, we need 
\begin{equation}
\label{eqn:N_3}
\le c_0^2 \min \{e^{p_1'-l_2},e^{(l_2-p_1)/2}\} \max\{1,e^{3(p_2-l_2)-(l_2'-l_2)}\}=:N_3
\end{equation}
many of these subsets. Continuing in this way at the end of the inductive step we consider the interval $[\max\{l_k,p_{k'}\},K+R]$. Assume that $\max\{l_k,p_{k'}\}=l_k$ so that $l_k'=K+R$ and $k'=k-1$ (the other case is similar and left to the reader).  We have the sets of the form
$$\T^{l_k}(x_0)u^+D_{\eta/2}^{U^+}\alpha^{-l_k} B_{\eta/2}^{U^-C}\alpha^{l_k}$$
that are obtained in the previous step. Let us divide them into $\lceil e^{3(l_k'-l_k)/2} \rceil^2$ small sets. By Lemma~\ref{lem:vectorconfig} we have $\le c_0 \min \{e^{p_{k-1}'-l_k},e^{(l_k-p_{k-1})/2}\}$ configurations and for each we apply Lemma~\ref{lem:volv} with $S'=S=l_k'-l_k$. Hence, we need to consider only 
\begin{equation}
\label{eqn:N_{2k-1}}
\le c_0^2\min \{e^{p_{k-1}'-l_k},e^{(l_k-p_{k-1})/2}\} e^{3(l_k'-l_l)-(l_k'-l_k)}=:N_{2k-1}
\end{equation}
of them. Thus, in the inductive step we divided the sets obtained earlier into
 $$\lceil e^{3(p_1-l_1)/2}\rceil^2 \lceil e^{3(l_2-p_1)/2}\rceil^2 \cdots \lceil e^{3(l_k'-l_k)/2}\rceil^2$$
 many parts and deduced that we only need to
 take
 \begin{equation}
 \le N_1 N_2 N_3 \cdots N_{2k-1}
 \end{equation}
  many of them where each set is of the form
$$\T^{K+R}(x_0)u^+D_{\eta/2}^{U^+}\alpha^{-K-R} B_{\eta/2}^{U^-C}\alpha^{K+R}.$$
On the other hand, let us multiply the $\max$ term of \eqref{eqn:N_1} with the $\min$ term of \eqref{eqn:N_2} to get
$$\max\{1,e^{3(p_1-l_1)-(l_1'-l_1)}\}  \min\{e^{(l_1'-p_1)/2},e^{p_1-l_1}\}.$$
If $\max\{1,e^{3(p_1-l_1)-(l_1'-l_1)}\}=e^{3(p_1-l_1)-(l_1'-l_1)}$ then clearly the multiplication above is $\le e^{3(p_1-l_1)-(l_1'-l_1)}e^{(l_1'-p_1)/2}\le e^{2(p_1-l_1)}. $ Otherwise, it is $\le e^{p_1-l_1}$. Thus, in either case we have
$$\le e^{2(p_1-l_1)}.$$
Similarly, let us multiply the $\max$ term of \eqref{eqn:N_2} with the $\min$ term of \eqref{eqn:N_3}
$$\max\{e^{3(l_2-p_1)/2},e^{3(l_2-p_1)-(p_1'-p_1)} \}\min \{e^{p_1'-l_2},e^{(l_2-p_1)/2}\}. $$
If $\max\{e^{3(l_2-p_1)/2},e^{3(l_2-p_1)-(p_1'-p_1)}\}=e^{3(l_2-p_1)-(p_1'-p_1)}$ then the above multiplication is $\le e^{3(l_2-p_1)-(p_1'-p_1)} e^{p_1'-l_2}=e^{2(l_2-p_1)}.$ Otherwise, it is 
$$\le e^{3(l_2-p_1)/2} e^{(l_2-p_1)/2} =e^{2(l_2-p_1)}.$$
Hence, in either case we have that the product is $\le e^{2(l_2-p_1)}.$

We continue in this way until we have considered all $\max$ and $\min$ terms. Thus, we obtain that
\begin{align*}
N_1 N_2 N_3 \cdots N_{2k-1} &\le c_0^{4k} e^{2(p_1-l_1)}e^{2(l_2-p_1)}\cdots e^{2(p_{k-1}-l_{k-1})}e^{2(l_k'-l_k)}\\
 &=c_0^{4k} e^{2(p_1-l_1)+2(l_2-p_1)+\cdots +2(l_k'-l_k)}\\
 &=c_0^{4k} e^{2|V_{n+1}|}
  \end{align*}
 
We know that $k$ is the number of elements of $\mathcal{L}$ restricted to the interval $V_{n+1}$. From Lemma~\ref{lem:card}  we have that $k \leq \frac{|V_{n+1}|}{\lfloor \log M\rfloor}+1$. Therefore, for the inductive step $K+|V_{n+1}|$, we get that the set  $Z_{\emph O}^+(V)$ can be covered by
 \begin{align*}
 &\ll e^{3K}e^{-(|V_1|+...+|V_{n}|)}c_0^{4\frac{|V_1|+...+|V_{n}|}{\lfloor \log M\rfloor}+4n+n'}e^{2|V_{n+1}|}c_0^{4\frac{|V_{n+1}|}{\lfloor \log M\rfloor}+4}\\
 &=e^{3(K+|V_{n+1})|}e^{-(|V_1|+...+|V_{n+1}|)}c_0^{4\frac{|V_1|+...+|V_{n+1}|}{\lfloor \log M\rfloor}+4(n+1)+n'}
 \end{align*}
 many forward Bowen $K+|V_{n+1}|$-balls.

 Hence, letting $K=N$ together with \eqref{eqn:boundform} we see that the set $Z_{\emph O}^+(V)$ can be covered by
 $\leq e^{3N-|V|}c_0^{\frac{4|V|}{\lfloor \log M\rfloor}+\frac{5N}{\lfloor \log M\rfloor}}\leq e^{3N-|V|}c_0^{\frac{9N}{\lfloor \log M\rfloor}}$ many forward Bowen N-balls. 
 \end{proof}

\section{Proof of Theorem~\ref{thm:main}}
\label{sec:mainthm}
Our main tool in proving Theorem~\ref{thm:main} will be Lemma~\ref{entropy}.
\begin{proof}[Proof of the Theorem~\ref{thm:main}] Note first that it is sufficient to
consider ergodic measures. For if $\mu$ is not ergodic, we can
write $\mu$ as an integral of its ergodic components $\mu=\int
\mu_t d\tau(t)$ for some probability space $(E,\tau)$, see for example \cite[Theorem~6.2]{EinWar}. Therefore,
we have $\mu(X_{\ge M})=\int \mu_t(X_{\ge M})d\tau(t)$, but also
$h_{\mu}(\T)=\int h_{\mu_t}(\T) d\tau(t)$, see for example \cite[Theorem~8.4]{WB}, so that the
desired estimate follows from the ergodic case.

Suppose that $\mu$ is ergodic. We would like to apply
Lemma~\ref{entropy}. For this we need to find an upper bound for
covering $\mu$-most of the space $X$ by Bowen $N$-balls. So, let $M\ge 100$ be such that $\mu(X_{\le M})>0$.
Thus, ergodicity of $\mu$ implies that $\mu(\bigcup_{k=0}^{\infty} \T^{-k}
X_{\le M})=1$. Hence, for every $\epsilon>0$ there is a constant $K \geq 1$
such that $Y=\displaystyle\bigcup_{k=0}^{K-1} \T^{-k} (X_{\le M})$
satisfies
$\mu(Y)>1-\epsilon.$\\
Moreover, the pointwise ergodic theorem implies
\[
\frac{1}{2N-1}\sum_{n=-N+1}^{N-1}1_{X_{\geq M}}(\T^n(x))\to \mu(X_{\geq M})
\]
as $N \to \infty$ for a.e. $x \in X$. Thus, given $\epsilon
>0$, there exists $N_0$ such that for $N>N_0$ the average on the
left will be bigger than $\mu(X_{\geq M})-\epsilon$ for any $x \in
X_1$ for some $X_1 \subset X$ with measure $\mu(X_1)> 1-\epsilon$.
Clearly, for any $N$ we have $\mu(Z)>1-2\epsilon$ where
\[Z=X_1\cap\T^NY.\]
Now, we would like to find an upper bound for the number of Bowen
$N$-balls needed to cover the set $Z$. Here $N \rightarrow \infty$
while $\epsilon$ and hence $K$ are fixed. Since
$Y=\displaystyle\bigcup_{k=0}^{K-1} \T^{-k} X_{\le M}$, we can
decompose $Z$ into $K$ sets of the form
\[Z'=X_1\cap\T^{N-k}X_{\le M}\]
but since $K$ is fixed, it suffices to find an upper bound for the
number of Bowen $N$-balls needed to cover one of these. Consider the set
$Z'$ which we split into the sets
$Z(\mathcal{N})$ as in Proposition~\ref{prop:main} (applied to the parameter $N-k$ instead of $N$) for the
various subsets $\mathcal{N} \in \mathcal{M}_{N-k}$. By
Lemma~\ref{lem:marked} we know that we need $\ll_M e^{\frac{10N\log \lfloor \log M\rfloor}{\lfloor \log M\rfloor}}$ many of these under the assumption that $M\ge 100>e^4$. Moreover, by our assumption on $X_1$ we only need to look at sets $V_x \subset
[-N+k+1,N-k-1]$ with $|V_x| \geq (\mu(X_{\geq M })-2\epsilon)(2N-1)$ (where we assume that $N$ is sufficiently large).
On the other hand, Proposition~\ref{prop:main}
gives that each of those sets $Z(\mathcal{N})$ can be covered by
$\leq e^{6N-|V_x|}c_0^{\frac{18N}{\lfloor \log M\rfloor}}$ Bowen $(N-k)$-balls for some constant $c_0>0$ that does not depend on $M$. It is easy to see from the definition that a Bowen $(N-k)$-ball can be covered by at most $c_1^k$ many Bowen $N$-balls. Together we see that $Z$ can be covered by
 $$\ll_{M,K} e^{\frac{10N\log \lfloor \log M\rfloor}{\lfloor \log M\rfloor}}c_0^{\frac{18N}{\lfloor \log M\rfloor}}e^{6N-(\mu(X_{\geq M })-2\epsilon)(2N-1)}$$ many Bowen
 $N$-balls. Applying Lemma~\ref{entropy} we arrive at
 \begin{eqnarray*}
  h_{\mu}(\T)&\leq
&\lim_{\epsilon \to 0}\liminf_{N \rightarrow
\infty}\frac{\log BC(N,\epsilon)}{2N} \\
&\leq &\lim_{\epsilon \to 0}  (3-(\mu(X_{\geq M })-2\epsilon)+O(\frac{\log  \log M}{ \log M})\\
&\le & 3-\mu(X_{\geq M })+O(\frac{\log  \log M}{ \log M})
\end{eqnarray*}
 which completes the proof for any sufficiently large $M$ with $\mu(X_{\le M})>0$. 
 However, we claim that the same conclusion holds for any sufficiently large $M$ independent of $\mu$ (which e.g.\ is crucial for proving Corollary~\ref{cor:lim}). 
 
 If $\mu(X_{\leq100})>0$ then the claim is true by the above discussion. So, assume that $\mu(X_{\leq 100})=0$ and let
 $$M_\mu=\inf\{M >100 : \mu(X_{\leq M})>0\}.$$
Since $\mu(X_{\le M})>0$ for any $M > M_\mu \ge 100$ we have by the discussion above
\begin{equation}
\label{eqn:h}
 h_{\mu}(\T) \leq 3-\mu(X_{\ge M })+O(\frac{\log \log M}{\log M}).
 \end{equation}

If $\mu(X_{\le M_{\mu}})>0$ then $\eqref{eqn:h}$ also holds for $M=M_{\mu}$ by the above. If on the other hand, $\mu(X_{\le M_{\mu}})=0$ then $\lim_{n \to \infty}\mu(X_{\ge M_{\mu}+\frac{1}{n}})=\mu(X_{>M_{\mu}})=\mu(X_{\ge M_{\mu}})$
and \eqref{eqn:h} for $M=M_{\mu}$ follows from \eqref{eqn:h} for $M=M_{\mu}+\frac{1}{n}.$ Since $\mu(X_{\ge M_{\mu}})=1$ this simplifies to
$$h_{\mu}(\T)\le 2+O(\frac{\log  \log M}{ \log M}).$$
Since $\frac{\log  \log M}{ \log M}$ is a decreasing function for $M\ge 100$ and $\mu(X_{\ge M})=1$ for $M \le M_{\mu}$ we obtain that \eqref{eqn:h} trivially also holds for any $M\in [100,M_{\mu}).$
  \end{proof}
\section{Limits of measures with high dimension}
\label{sec:locdim}
In this section we prove Theorem~\ref{thm:dim} and Corollary~\ref{cor:sing2}. Our main tool is a version of Proposition~\ref{prop:main}. Let $N,M>0$ be given. For any $x$ we define $V_{x}\in [0,N-1]$ to be the set of times $n \in [0,N-1]$ for which $\T^n(x) \in X_{\ge M}.$ 
Now, Proposition~\ref{prop:main} can be rephrased as follows.
\begin{prop}
\label{prop:main'}
For a fixed set $\mathcal{N}=\mathcal{N}_{[0,N-1]}(x_0)$ of labeled marked times in $[0,N-1]$ we have that the set
 $$Z^+(\mathcal{N})=\{x \in X_{\le M} : \mathcal{N}_{[0,N-1]}(x)=\mathcal{N}_{[0,N-1]} \}$$
can be covered by $\ll_M e^{3N-|V_{x_0}|}c_0^{\frac{9N}{\lfloor \log M \rfloor}}$ many sets of the form
$$\T^{-N}(\T^N(x)u^+)D_{\frac{\eta}{2}e^{-3N/2}}^{U^+} B_{\frac{\eta}{2}}^{U^-C}.$$
\end{prop}
\begin{proof}
In the proof of Proposition~\ref{prop:main} we inductively proved that the set 
$$Z_{\it O}^+=\{x \in {\it O} : \mathcal{N}_{[0,N-1]}(x)=\mathcal{N}_{[0,N-1]} \}$$ can be covered by
$e^{3N-|V_{x_0}|} c_0^{\frac{9N}{\lfloor \log M \rfloor}}$ many pre-images under $\T^N$ of sets of the form
$$\T^N(x_0)u^+D_{\eta/2}^{U^+}\alpha^{-N} B_{\eta/2}^{U^-C}\alpha^N.$$
So, $Z_{\it O}^+$ can be covered by the sets of the form
$$\T^{-N}(\T^N(x_0)u^+)\alpha^N D_{\eta/2}^{U^+}\alpha^{-N} B_{\eta/2}^{U^-C}.$$
This completes the proof since we have $\alpha^N D_{\eta/2}^{U^+}\alpha^{-N}=D_{\frac{\eta}{2}e^{-3N/2}}^{U^+}$ and since $X_{\le M}$ is compact.
\end{proof}
In the following let $\nu$ be a probability measure on $X$ which has a dimension at least $d$ in the unstable direction (see \eqref{eqn:unstable}). We wish to prove Theorem~\ref{thm:dim}.

 For any $\kappa>0$ small we are interested in the upper estimate for
 $$\nu(\{x \in X_{<M} : |V_x|>\kappa N\}).$$
Proposition~\ref{prop:main'} together with Lemma~\ref{lem:marked} gives the following.
\begin{lem}
\label{lem:kappa}
For any $N>0$ large, we have 
 $$\nu(\{x \in X_{\le M} : |V_x|>\kappa N\}) \ll_{M,\delta}  e^{\frac{6-2\kappa-3d+3\delta}{2}N+ \frac{9N\log (c_0\log M)}{ \log M }}.$$
\end{lem}
\begin{proof}
From Lemma~\ref{lem:marked} we know that the set $X_{<M}$ can be decomposed into 
$$\ll_M e^{\frac{5 N\log \lfloor \log M\rfloor}{\lfloor \log M\rfloor}}$$
 many sets of the form $Z^+(\mathcal{N})$. We are only interested in those sets of marked times $\mathcal{N}_{[0,N-1]}(x)$ for which $|V_x|>\kappa N$. On the other hand, from Proposition~\ref{prop:main'} we know that such sets can be covered by 
 $e^{(3-\kappa)N} c_0^{\frac{9N}{\lfloor \log M \rfloor}}$
 many sets of the form
 $$\T^{-N}(\T^N(x)u^+)D_{\frac{\eta}{2}e^{-3N/2}}^{U^+} B_{\frac{\eta}{2}}^{U^-C}.$$
 However, from the assumption on dimension of the measure $\nu$ we have
 $$\nu(\T^{-N}(\T^N(x)u^+)D_{\frac{\eta}{2}e^{-3N/2}}^{U^+} B_{\frac{\eta}{2}}^{U^-C}) \ll_\delta (\frac{\eta}{2}e^{-3N/2})^{d-\delta}$$
 once $N$ is sufficiently large. Thus,
 $$ \nu(\{x \in X_{\le M} : |V_x|>\kappa N\}) \ll_{M,\delta} e^{\frac{5 N\log \lfloor \log M\rfloor}{\lfloor \log M\rfloor}}e^{(3-\kappa)N} c_0^{\frac{9N}{\lfloor \log M \rfloor}}(\frac{\eta}{2}e^{-3N/2})^{d-\delta}.$$
 This simplifies to
$$\nu(\{x \in X_{\le M} : |V_x|>\kappa N\}) \ll_{M,\delta} e^{\frac{6-2\kappa-3d+3\delta}{2}N+ \frac{9N\log (c_0\log M)}{ \log M }}.$$
\end{proof}
\begin{proof}[Proof of Theorem~\ref{thm:dim}]
Note that for $d \le 4/3$ the conclusion in the theorem is trivial. Hence we assume that $d >4/3$. In order to prove Theorem~\ref{thm:dim} we need to estimate an upper bound for $\mu_N(X_{\ge M})$ for $M, N$ large. Let us recall that 
$$\mu_N=\frac{1}{N}\sum_{i=0}^{N-1}\T^i_*\nu.$$
Hence, 
\begin{align*}
\mu_N(X_{\geq M})&=\frac{1}{N}\sum_{n=0}^{N-1}\nu(\T^{-n}(X_{\geq M}))\\
&=\frac{1}{N}\sum_{n=0}^{N-1}\nu(X_{\le M} \cap \T^{-n}(X_{\geq M}))+\frac{1}{N}\sum_{n=0}^{N-1}\nu(X_{> M} \cap \T^{-n}(X_{\geq M})).
\end{align*}
However, we have $\nu(X_{> M})<\epsilon(M)$ where $\epsilon(M) \to 0$ as $M \to \infty$.
Hence,
\begin{equation}
\label{eqn:muN}
\mu_N(X_{\geq M}) \leq \epsilon(M) + \frac{1}{N}\sum_{n=0}^{N-1}\nu(X_{\le M} \cap \T^{-n}(X_{\geq M})).
\end{equation}
Thus, all we need to estimate is $\frac{1}{N}\sum_{n=0}^{N-1}\nu(X_{\le M} \cap \T^{-n}(X_{\geq M}))$. 

Now, recalling that $V_x=\{n \in [0,N-1] : \T^n(x) \in X_{\ge M}\}$ we note that
\begin{multline*}
\frac{1}{N}\sum_{n=0}^{N-1}\nu(X_{\le M} \cap \T^{-n}(X_{\geq M}))\\
=\frac{1}{N}\sum_{n=0}^{N-1} \sum_{W \subset [0,N]} \nu(\{x \in X_{\le M}: V_x=W\}\cap \T^{-n}(X_{\ge M})),
\end{multline*}
where $\nu(\{x \in X_{\le M}: V_x=W\}\cap \T^{-n}(X_{\ge M}))$ is either 0 or $\nu(\{x \in X_{\le M}: V_x=W\})$. Therefore, we switch the order of summation and get
\begin{align*}
&=\frac{1}{N}\sum_{W \subset [0,N-1]}|W|\nu(\{x \in X_{\le M}: V_x=W\})\\
&=\frac{1}{N}\sum_{i=1}^{N}i\nu(\{x \in X_{\le M}: |V_x|=i\})\\
&= \frac{1}{N}\sum_{i=1}^{\lfloor \kappa N \rfloor}i\nu(\{x \in X_{\le M}: |V_x|=i\})+\frac{1}{N}\sum_{i=\lceil \kappa N \rceil }^{N}i \nu(\{x \in X_{\le M}: |V_x|=i\})\\
&\le \frac{1}{N}\lfloor \kappa N \rfloor\nu(X_{\le M})+\frac{1}{N}N\nu(\{x \in X_{\le M}: |V_x|>\kappa N\})
\end{align*}
Let $K(M,\delta)>0$ be the implicit constant that appeared in Lemma~\ref{lem:kappa}. Then using Lemma~\ref{lem:kappa} we obtain
$$\frac{1}{N}\sum_{n=0}^{N-1}\nu(X_{<M} \cap \T^{-n}(X_{\geq M}))\le \kappa+K(M,\delta)e^{\frac{6-2\kappa-3d+3\delta}{2}N+ \frac{9N\log (c_0\log M)}{ \log M }}.$$
Thus, together with \eqref{eqn:muN} we get
\begin{equation}
\mu_N(X_{\ge M})\le \epsilon(M)+\kappa+K(M,\delta)e^{(\frac{6-2\kappa-3d+3\delta}{2}+ \frac{9\log (c_0\log M)}{ \log M })N}.
\end{equation}
By assumption we have $d>\frac{4}{3}$. Let $\kappa > \frac{6-3d}{2}$ (which we will later choose to approach $\frac{6-3d}{2}$). Now, we let $\delta>0$ to be small enough so that
$$6-2\kappa-3d+3\delta<0.$$
Let $\epsilon>0$ be given. For $M$ sufficiently large we can make sure that $\epsilon(M)<\epsilon/2$ and that $\frac{6-2\kappa-3d+3\delta}{2}+ \frac{9\log (c_0\log M)}{ \log M }<0.$
 Thus, 
 $$K(M,\delta)e^{(\frac{6-2\kappa-3d+3\delta}{2}+ \frac{9\log (c_0\log M)}{ \log M })N} \to 0$$
  as $N\to \infty$. So, we conclude that for $N$ large enough we get
  $$\mu_N(X_{\ge M})\le \kappa+\epsilon$$
  which gives in the limit that $\mu(X)>1- \kappa.$ This is true for any $\kappa >\frac{6-3d}{2}$. Thus,
  $$\mu(X)\ge 1-\frac{6-3d}{2}=\frac{3d-4}{2}.$$
\end{proof}
Next, we prove Corollary~\ref{cor:sing2}. We need the following Corollary 4.12 from \cite{Fal}.
\begin{thm}
\label{thm:Fal}
Let $F$ be a Borel subset of $\R^n$ with $0<\mathcal{H}^s(F) \le \infty.$ Then there is a compact set $E \subset F$ such that $0<\mathcal{H}^s(E)<\infty$ and a constant $b$ such that
$$\mathcal{H}^s(E\cap B_\delta({\bf r}))\le b \delta^s$$
for all ${\bf r} \in \R^n$ and $\delta>0.$
\end{thm}

\begin{proof}[Proof of Corollary~\ref{cor:sing2}]
As any divergent point is also divergent on average, we get from \cite[Corollary 1.2]{Che} that the set of points $F_0\subset X$ that are divergent on average has at least dimension $\frac43+6$. 
So assume now that the Hausdorff dimension of $F_0$ is greater than $\frac{4}{3}+6.$ Then, by the behavior of Hausdorff dimension under countable unions, there is some subset $F\subset F_0$ with compact closure and small diameter for which the Hausdorff dimension is also bigger than $\frac43+6$.  Here we may assume that $F=F_0\cap (x_0D_\eta B_\eta^{U^-C})$ and that $x_0D_\eta B_\eta^{U^-C}$ is the injective image of the corresponding set in $\SL_3(\R)$. It then follows that $F=x_0 D' B_\eta^{U^-C}$
and that $D'$ has Hausdorff dimension bigger than $\frac{4}3$. Thus, for sufficiently small $\epsilon>0$ we have that $\mathcal{H}^{\frac{4}{3}+\epsilon}(D')=\infty.$ We may identify $U^+$ with $\R^2$ and apply Theorem~\ref{thm:Fal}. Therefore, there exists a compact set $E \subset D'$ such that $0<\mathcal{H}^{\frac{4}{3}+\epsilon}(E)<\infty$ and a constant $b$ such that
$$\mathcal{H}^{\frac{4}{3}+\epsilon}(E\cap B_\delta({\bf r}))\le b \delta^{\frac{4}{3}+\epsilon}$$
for all ${\bf r} \in \R^2$ and $\delta>0.$ We define $\nu_0=\frac{1}{\mathcal{H}^{\frac{4}{3}+\epsilon}(E)}\mathcal{H}^{\frac{4}{3}+\epsilon} _{|_E}$ so that $\nu_0(U^+)=1$.
Let $\tau$ be the map from $U^+$ to $X$ defined by $\tau(u)=x_0u.$ Now, we let $\nu=\tau_*\nu_0$ to be the push-forward of the measure $\nu_0$ under the map $\tau$. It follows that for any $\delta>0$ and for any $x\in X$ we have
$$\nu(xB_{\delta}^{U^+}B_\eta^{U^-C})\ll \delta^{\frac{4}{3}+\epsilon}. $$
Now, if we define $\mu_N$ as before then Theorem~\ref{thm:dim} implies that
the limit measure $\mu$ has at least $\frac{3}{2}(\frac{4}{3}+\epsilon-\frac{4}{3})\frac{3\epsilon}{2}>0$ mass left. However, 
the assumption on $F_0$ and dominated convergence applied to
\[
 \mu_N(X_{\le M})=\int\frac1N\sum_{n=0}^{N-1}\chi_{T^{-n}X_{\le M}} d\nu
\]
implies that $\mu_N(X_{\le M})\to 0$ as $N\to\infty$ for any fixed $M$. This gives a contradiction and the corollary.
\end{proof}
\bibliographystyle{plain}
\bibliography{SL3}
\end{document}